\documentclass[tocmacros]{daj}

\dajAUTHORdetails{%
  title = {Bounds for the Local Properties Problem for Difference Sets}, 
  author = {Sanjana Das},
  plaintextauthor = {Sanjana Das},
}   

\dajEDITORdetails{%
   year={2025},
   number={10},
   received={10 December 2023},   
   published={3 September 2025},  
   doi={10.19086/da.142094},       
}   

\usepackage{tikz}
\usetikzlibrary{arrows.meta}
\usepackage[shortlabels]{enumitem}

\newcommand{\floor}[1]{\left\lfloor #1 \right\rfloor}
\newcommand{\ceil}[1]{\left\lceil #1 \right\rceil}

\newcommand{\abs}[1]{\left\lvert #1 \right\rvert}
\newcommand{\EE}{\mathbb E}

\newcommand{\RR}{\mathbb R}
\newcommand{\cC}{\mathcal C}
\newcommand{\cG}{\mathcal G}
\newcommand{\cS}{\mathcal S}
\newcommand{\cT}{\mathcal T}
\newcommand{\eps}{\varepsilon}
\newcommand{\ol}{\overline}
\DeclareMathOperator{\Sol}{Sol}

\definecolor{DarkSlateGray3}{RGB}{121, 205, 205}
\definecolor{MediumPurple3}{RGB}{137, 104, 205}
\definecolor{Honeydew3}{RGB}{193, 205, 193}
\definecolor{Honeydew4}{RGB}{131, 139, 131}

\begin{document}

\begin{frontmatter}[classification=text]

\title{Bounds for the Local Properties Problem for Difference Sets} 

\author[sdas]{Sanjana Das}

\begin{abstract}
  We consider the local properties problem for difference sets: we define $g(n, k, \ell)$ to be the minimum value of $\abs{A - A}$ over all $n$-element sets $A \subseteq \RR$ with the `local property' that $\abs{A' - A'} \geq \ell$ for all $k$-element subsets $A' \subseteq A$. We view $k$ and $\ell$ as fixed, and study the asymptotic behavior of $g(n, k, \ell)$ as $n \to \infty$. One of our main results concerns the quadratic threshold, i.e., the minimum value of $\ell$ such that $g(n, k, \ell) = \Omega(n^2)$; we determine this value exactly for even $k$, and we determine it up to an additive constant for odd $k$. We also show that for all $1 < c \leq 2$, the `threshold' for $g(n, k, \ell) = \Omega(n^c)$ is quadratic in $k$; conversely, for $\ell$ quadratic in $k$, we obtain upper and lower bounds of the form $n^c$ for (not necessarily equal) constants $c > 1$. In particular, this provides the first nontrivial upper bounds in the regime where $\ell$ is quadratic in $k$. 
\end{abstract}
\end{frontmatter}


\section{Introduction}

\subsection{History of local properties problems} \label{subsec:history}

There is a long history of studying objects that have local properties, and trying to understand the extent to which these local properties determine global properties. In particular, many questions of this form concern going from local to global structure: if we know every small piece of an object is `unstructured' in some sense, how unstructured does the \emph{entire} object have to be? 

For example, Erd\H{o}s and Shelah \cite[Section V]{Erd86} initiated the study of a local properties problem for graphs, where we consider edge-colorings of a complete graph and think of an edge-coloring as `unstructured' if it contains many distinct colors. For positive integers $n$, $k$, and $\ell$, we define $f(n, k, \ell)$ to be the minimum number of colors in an edge-coloring of $K_n$ with the property that every induced subgraph $K_k$ contains at least $\ell$ colors. We think of $k$ and $\ell$ as fixed and $n$ as large, and we wish to understand the asymptotic behavior of $f(n, k, \ell)$ as $n$ tends to infinity. There are two natural ways to approach this problem. One is to fix $k$ and $\ell$ and attempt to bound $f(n, k, \ell)$. The other is to attempt to understand various thresholds for $f(n, k, \ell)$ --- given $k$, we search for the smallest value of $\ell$ (as a function of $k$) for which we have a certain lower bound on $f(n, k, \ell)$. 

This local properties problem for graphs has been studied by many authors: Erd\H{o}s and Gy\'arf\'as \cite{EG97} established the linear and quadratic thresholds (the smallest values of $\ell$ as a function of $k$ for which $f(n, k, \ell) = \Omega(n)$ and $f(n, k, \ell) = \Omega(n^2)$, respectively) and an upper bound on the polynomial threshold (the smallest value of $\ell$ for which $f(n, k, \ell) = \Omega(n^\eps)$ for some $\eps > 0$); Conlon, Fox, Lee, and Sudakov \cite{CFLS15} later proved a matching lower bound, thus establishing the polynomial threshold. Erd\H{o}s and Gy\'arf\'as \cite{EG97} also proved a general upper bound for $f(n, k, \ell)$ using a probabilistic construction. Pohoata and Sheffer \cite{PS19} and Fish, Pohoata, and Sheffer \cite{FPS20} later proved lower bounds for $f(n, k, \ell)$ in several regimes, which grow arbitrarily close to the upper bound of Erd\H{o}s and Gy\'arf\'as as $k$ grows large. A more detailed discussion of this problem can be found in \cite{FPS20} and \cite{PS19}. 

Erd\H{o}s \cite{Erd86} also posed a similar local properties problem for distinct distances, to determine the minimum number of distinct distances spanned by a set of $n$ points in $\RR^2$ such that every $k$ points span at least $\ell$ distinct distances. This problem has also been studied by various authors; more discussion of this problem can be found in \cite{FLS19}.

\subsection{A local properties problem for difference sets}

We study an arithmetic version of such local properties problems, suggested by Zeev Dvir (see \cite{PS19}). We consider sets of real numbers for which every small piece is `arithmetically unstructured,' and we wish to understand how arithmetically unstructured this local property requires the entire set to be. One measure of the extent to which a set is arithmetically structured is the size of its difference set --- roughly, sets with smaller difference set relative to their size possess more arithmetic structure. So we consider sets for which every small subset has a reasonably large difference set, and attempt to understand how large the difference set of the entire set must be. 

We now make this precise. For a set $A \subseteq \RR$, we define the \emph{difference set} of $A$ to be \[A - A = \{\abs{a - b} \mid a, b \in A, \, a \neq b\}.\] Note that we include only positive differences in $A - A$; this definition is somewhat nonstandard but more natural for the problem we study. (When discussing the differences present in a set, we refer only to positive differences, unless otherwise stated.)

For integers $n \geq k \geq 1$ and $\ell \geq 0$, we define $g(n, k, \ell)$ to be the minimum value of $\abs{A - A}$ over all $n$-element sets $A \subseteq \RR$ with the property that every $k$-element subset $A' \subseteq A$ satisfies $\abs{A' - A'} \geq \ell$; we refer to this property as the \emph{$(k, \ell)$-local property}. (Note that the same definition can be used to define $g(n, k, \ell)$ even when $\ell$ is not a nonnegative integer; then $g(n, k, \ell) = g(n, k, \ceil{\ell})$. This makes certain results slightly more convenient to state.) We think of $k$ and $\ell$ as fixed and $n$ as large, and we study the asymptotic behavior of $g(n, k, \ell)$ as $n$ tends to infinity. (In particular, in all asymptotic notation in this paper, $n$ tends to infinity while all other quantities are fixed; the implied constants may depend on $k$, $\ell$, and any other stated parameters.)

As with the local properties problem for graph colorings, we can attempt to understand $g(n, k, \ell)$ either by fixing $\ell$, or by considering various thresholds. Note that for all $n$-element sets $A \subseteq \RR$ we have $n - 1 \leq \abs{A - A} \leq \binom{n}{2}$ (and $\abs{A - A} = n - 1$ if and only if $A$ is an arithmetic progression), so for all $\ell \leq \binom{k}{2}$ we have \[n - 1 \leq g(n, k, \ell) \leq \binom{n}{2}.\] (If $\ell > \binom{k}{2}$, then no set $A$ can satisfy the $(k, \ell)$-local property, so we say that $g(n, k, \ell) = +\infty$.) Also note that $g(n, k, \ell)$ is weakly increasing in $\ell$, as increasing $\ell$ only makes the $(k, \ell)$-local property stronger. So it is natural to consider, for fixed $k$, how long $g(n, k, \ell)$ stays linear and when it becomes quadratic as we increase $\ell$. To formalize this, as defined in \cite{Li22}, the \emph{superlinear threshold} is defined to be the largest integer $\ell$ (as a function of $k$) such that $g(n, k, \ell) = O(n)$, and the \emph{quadratic threshold} is defined to be the smallest integer $\ell$ (as a function of $k$) such that $g(n, k, \ell) = \Omega(n^2)$. (Note that $g(n, k, k - 1) = n - 1$, so at the superlinear threshold we actually have $g(n, k, \ell) = \Theta(n)$; similarly, if $k \geq 4$ then $g(n, k, \binom{k}{2}) = \binom{n}{2}$, so at the quadratic threshold we actually have $g(n, k, \ell) = \Theta(n^2)$.)  

\subsection{Previous bounds}

This problem was first studied by Fish, Pohoata, and Sheffer \cite{FPS20}, who proved the following family of lower bounds: for all $r \geq 2$ and $k$ divisible by $2r$, we have 
\begin{align}
    g\left(n, k, \binom{k}{2} - \frac{(r - 1)(r + 2)}{2}\binom{k/r}{2} + 1\right) = \Omega(n^{\frac{r}{r - 1}\cdot \frac{k - 2r}{k}}).\label{eqn:fps}
\end{align} 
For example, taking $r = 2$, this gives that for $4 \mid k$ we have $g(n, k, \frac{k^2}{4} + 1) = \Omega(n^{2 - \frac{8}{k}})$; more generally, if we think of $r$ as fixed and $k$ as large, this provides a family of lower bounds for $g(n, k, \ell)$ for values of $\ell$ that are quadratic in $k$, with leading coefficient $\frac{1}{2} - \frac{(r - 1)(r + 2)}{4r^2}$ (which ranges from $\frac{7}{32}$ to $\frac{1}{4}$). Fish, Pohoata, and Sheffer obtained these bounds using the technique of color energy, which they had developed to study the local properties problem for graph colorings. 

Fish, Pohoata, and Sheffer also proved the following upper bound for $g(n, k, \ell)$ (for certain values of $\ell$): for any $\eps > 0$, there exists $c > 0$ such that for all sufficiently large $k$ we have \[g(n, k, ck(\log k)^{\frac{1}{4} - \eps}) = n\cdot 2^{O(\sqrt{n\log n})}.\]
The idea of their proof is to consider $3$-term arithmetic progressions --- results from additive combinatorics give that for such values of $k$ and $\ell$, any $k$-element set with fewer than $\ell$ differences must contain a $3$-term arithmetic progression. Then any set $A$ with no $3$-term arithmetic progressions satisfies the $(k, \ell)$-local property, and Behrend's construction in \cite{Beh46} provides such sets $A$ with reasonably small $\abs{A - A}$. 

Fish, Lund, and Sheffer \cite{FLS19} then proved the upper bound \[g\left(n, k, \frac{k^{\log_2 3} - 1}{2}\right) = O(n^{\log_2 3}).\] Their construction for the set $A$ was roughly an affine $t$-cube (i.e., a set of the form $\{a + \eps_1d_1 + \cdots + \eps_t d_t \mid \eps_1, \ldots, \eps_t \in \{0, 1\}\}$ for fixed $a$, $d_1$, \ldots, $d_t$) with $t \approx \log_2 n$. 

This problem was then studied by Li \cite{Li22}, who proved that the superlinear threshold is $k - 1$ by showing that $g(n, k, k) = \omega(n)$, i.e., if $A$ is an $n$-element set with no $k$-term arithmetic progression, then we must have $\abs{A - A} = \omega(n)$. 

Li also proved several lower bounds for $g(n, k, \ell)$: if $k$ is even, then \begin{align}g\left(n, k, \frac{3k^2}{8} - \frac{3k}{4} + 2\right) = \Omega(n^2);\label{eqn:li-2-bound}\end{align} if $k$ is divisible by $8$, then \begin{align}g\left(n, k, \frac{9k^2}{32} - \frac{9k}{16} + 5\right) = \Omega(n^{4/3});\label{eqn:li-4-3-bound}\end{align} and if $k \geq 8$ is a power of $2$, then \begin{align}g\left(n, k, \frac{k^{\log_2 3} + 1}{2}\right) = \Omega(n^{1 + \frac{2}{k - 2}}).\label{eqn:li-huge-cube}\end{align} In particular, \eqref{eqn:li-2-bound} implies that the quadratic threshold is at most $\frac{3k^2}{8} - \frac{3k}{4} + 2$. The proofs of all three of these bounds can be interpreted as showing that a set with sufficiently small difference set must contain $\frac{k}{2^t}$ disjoint congruent affine $t$-cubes (where $t = 1$, $2$, and $\log_2 k - 1$, respectively), whose elements form a $k$-element subset contradicting the $(k, \ell)$-local property for the stated values of $\ell$.

Li also proved, using a random construction, that for all $c \geq 2$ and $k$ large with respect to $c$, we have \[g(n, k, ck + 1) = O(n^{1 + \frac{c^2 + 1}{k}}).\]

\subsection{Our results}\label{subsec:results}

In this paper, we prove a family of upper and lower bounds for $g(n, k, \ell)$. These bounds have several consequences that allow us to better understand the behavior of $g(n, k, \ell)$; we state these consequences first, as they are more intuitive. 

First, we establish the quadratic threshold for all even $k$. 

\begin{corollary}\label{cor:even-qt}
    For all even $k$, we have \[g\left(n, k, \frac{k^2}{4}\right) = o(n^2) \text{ and } g\left(n, k, \frac{k^2}{4} + 1\right) = \Omega(n^2).\] In particular, the quadratic threshold is $\frac{k^2}{4} + 1$. 
\end{corollary}

For odd $k$, we are not able to determine the exact value of the quadratic threshold, but we obtain upper and lower bounds for the quadratic threshold that differ by $3$. 

\begin{corollary}\label{cor:odd-qt}
    For all odd $k$, we have \[g\left(n, k, \frac{(k + 1)^2}{4} - 4\right) = o(n^2) \text{ and } g\left(n, k, \frac{(k + 1)^2}{4}\right) = \Omega(n^2).\] In particular, the quadratic threshold is between $\frac{(k + 1)^2}{4} - 3$ and $\frac{(k + 1)^2}{4}$, inclusive. 
\end{corollary}

We also obtain bounds on the thresholds for other exponents, although these bounds are less tight --- for each $1 < c \leq 2$, we define the \emph{threshold for $\Omega(n^c)$} to be the smallest integer $\ell$ such that $g(n, k, \ell) = \Omega(n^c)$. (The quadratic threshold is the special case of this definition where $c = 2$.) We obtain that for each fixed $c$, the threshold for $\Omega(n^c)$ is quadratic in $k$. 

\begin{corollary}\label{cor:nc-threshold}
    For each $1 < c \leq 2$, there exist constants $0 < a_1 < a_2 < \frac{1}{2}$ (which can be defined as explicit functions of $c$ such that $a_1, a_2 \to 0$ as $c \to 1$) such that the following holds for all sufficiently large $k$: we have $g(n, k, a_1k^2) = o(n^c)$ and $g(n, k, a_2k^2) = \Omega(n^c)$. In particular, the threshold for $\Omega(n^c)$ is between $a_1k^2$ and $a_2k^2$. 
\end{corollary}

We can also view our results in terms of bounds on $g(n, k, \ell)$ for given $\ell$ --- when $\ell$ is quadratic in $k$ with leading coefficient less than $\frac{1}{4}$, we obtain both upper and lower bounds on $g(n, k, \ell)$ of the form $n^c$ for $1 < c < 2$ (though not with the same exponent). In particular, these bounds provide the first nontrivial upper bounds on $g(n, k, \ell)$ for values of $\ell$ that are quadratic in $k$ (the previously known upper bounds work with much smaller values of $\ell$). 

\begin{corollary}\label{cor:poly-bounds}
    For each $0 < a < \frac{1}{4}$, there exist $1 < c_1 < c_2 < 2$ (which can be defined as explicit functions of $a$ such that $c_1, c_2 \to 1$ as $a \to 0$) such that for all sufficiently large $k$, we have $g(n, k, ak^2) = \Omega(n^{c_1})$ and $g(n, k, ak^2) = o(n^{c_2})$. 
\end{corollary}

Finally, for each $k$, one can consider the set of `possible exponents' that occur in the asymptotics for $g(n, k, \ell)$ over all $\ell$ --- to make this precise, we define \[S_k = \left\{\liminf_{n \to \infty} \frac{\log g(n, k, \ell)}{\log n} \Bigm\vert 0 \leq \ell \leq \binom{k}{2}\right\}.\] As of now, very little is known about this set, beyond the fact that it contains $1$ for all $k$ and $2$ for all $k \geq 4$. However, we are able to show that $\abs{S_k}$ grows arbitrarily large as $k$ does. 

\begin{corollary}\label{cor:sk-loglogk}
    There exists an absolute constant $a > 0$ such that for all $k \geq 2$, we have $\abs{S_k} \geq a\log\log k$. 
\end{corollary}

We now state our more general theorems from which we deduce these statements. First, we prove the following family of upper bounds for $g(n, k, \ell)$. 

\begin{theorem}\label{thm:upper-bound}
    For all positive integers $k$ and all $1 < c \leq 2$, we have \[g\left(n, k, \ceil{\frac{(c - 1)(k - 1)}{c}}^2\right) = o(n^c).\] 
\end{theorem}

In order to obtain our lower bound on the quadratic threshold for odd $k$ (as stated in Corollary \ref{cor:odd-qt}), we also prove the following improvement on Theorem \ref{thm:upper-bound} in the case where $c = 2$ and $k$ is odd. 

\begin{proposition}\label{prop:odd-upper-bound}
    For all odd positive integers $k$, we have \[g\left(n, k, \frac{(k + 1)^2}{4} - 4\right) = o(n^2).\] 
\end{proposition}

We obtain Theorem \ref{thm:upper-bound} and Proposition \ref{prop:odd-upper-bound} by using a random construction. To ensure that the constructed set satisfies the $(k, \ell)$-local property for the stated values of $\ell$, we consider all possible `configurations' of equal differences that can be formed by a $k$-element subset. We first show, very roughly, that certain `bad' configurations are expected to occur very few times in the random construction, so using the alteration method, we can ensure that our set only contains `good' configurations. We then show that any $k$-element subset forming a good configuration must contain at least $\ell$ distinct differences. The technical steps of our proof rely on considering the systems of equations associated with such configurations and analyzing their linear algebraic properties. 

We also prove the following family of lower bounds for $g(n, k, \ell)$. 

\begin{theorem}\label{thm:lower-bound}
    For all positive integers $t$ and $k$ with $2^t \mid k$, we have \[g\left(n, k, \frac{3^{t - 1}}{4^t}k^2 + \frac{3^{t - 1} + 1}{2}\right) = \Omega(n^{1 + \frac{1}{2^t - 1}}).\] 
\end{theorem}

Using certain straightforward monotonicity properties of $g(n, k, \ell)$, we can extend Theorem \ref{thm:lower-bound} to all values of $k$ (not just multiples of $2^t$), with an adjustment in the value of $\ell$ that is linear in $k$; we state this extension in Corollary \ref{cor:lower-for-nonmultiples}.

We obtain Theorem \ref{thm:lower-bound} by showing that a set with small difference set must contain a \emph{specific} $k$-element configuration that has few distinct differences. This idea is similar to Li's proofs of \eqref{eqn:li-2-bound}, \eqref{eqn:li-4-3-bound}, and \eqref{eqn:li-huge-cube}; instead of using a configuration consisting of $\frac{k}{2^t}$ congruent affine $t$-cubes, we use one consisting of $\frac{k}{2^{t - 1}}$ congruent affine $(t - 1)$-cubes whose centers form $\frac{k}{2^t}$ pairs with equal sums. (In particular, to obtain the lower bound of $\Omega(n^2)$, we use a configuration consisting of $\frac{k}{2}$ pairs with equal sums instead of $\frac{k}{2}$ pairs with equal differences.) If we think of $t$ as fixed and $k$ as large, this improves the value of $\ell$ in the bounds by a factor of around $\frac{2}{3}$.  

If we consider $c$ (in Theorem \ref{thm:upper-bound}) and $t$ (in Theorem \ref{thm:lower-bound}) to be fixed and $k$ to be large, then Theorems \ref{thm:upper-bound} and \ref{thm:lower-bound} provide upper and lower bounds on $g(n, k, \ell)$ with constant exponents of $n$ (not depending on $k$), for values of $\ell$ that are quadratic in $k$ (with coefficients depending on the exponent in the bound). The relationships between the leading coefficient in $\ell$ and the exponent of $n$ in our two bounds are illustrated in Figure \ref{fig:bounds-graph}, with the blue line representing the upper bounds and the purple line representing the lower bounds. Although there is a gap between the two lines below the point $(\frac{1}{4}, 2)$, both approach the point $(0, 1)$ --- for the blue line, this is because we are able to obtain an upper bound of $n^c$ for every $1 < c \leq 2$, and the corresponding leading coefficient $(\frac{c - 1}{c})^2$ tends to $0$ as $c \to 1$; for the purple line, this is because in our lower bounds, the leading coefficient $\frac{3^{t - 1}}{4^t}$ tends to $0$ as $t \to \infty$, and the corresponding exponent $1 + \frac{1}{2^t - 1}$ tends to $1$. Very roughly, this feature of the graph, along with the fact that such a graph can be made in the first place, is what allows us to deduce Corollaries \ref{cor:nc-threshold} and \ref{cor:poly-bounds}.

\begin{figure}[ht]
    \centering 
    \begin{tikzpicture}[scale = 1.1]
        
        \fill [DarkSlateGray3!15] (12, 3.75) -- (6, 3.75) [domain = 2 : 1, samples = 20] plot({24 * (\x - 1)^2/(\x^2)}, {3.5*(\x - 1) + 0.25}) -- (0, 0) -- (12, 0) -- (12, 3.75);

        \foreach \i in {2, ..., 10} {;
            \fill [MediumPurple3!50!DarkSlateGray3!30] (0, {3.5/(2^\i - 1) + 0.25}) -- ({24/3*(3/4)^(\i - 1)}, {3.5/(2^\i - 1) + 0.25}) -- ({24/3*(3/4)^(\i - 1)}, 3.75) -- (0, 3.75) -- cycle;
        }
        \fill [MediumPurple3!50!DarkSlateGray3!30]  ({24/3*(3/4)^10}, {3.5/(2^10 - 1) + 0.25}) -- (0, 0.25) -- (0, 3.75) -- ({24/3*(3/4)^10}, 3.75) -- cycle;
        \draw [MediumPurple3, very thick] ({24/3*(3/4)^10}, {3.5/(2^10 - 1) + 0.25}) -- (0, 0.25);
        \foreach \i in {2, ..., 10} {
            \draw [MediumPurple3, very thick] ({24/3*(3/4)^\i}, {3.5/(2^\i - 1) + 0.25}) -- ({24/3*(3/4)^(\i - 1)}, {3.5/(2^\i - 1) + 0.25}) -- ({24/3*(3/4)^(\i - 1)}, {7/(2^\i - 2) + 0.25});
        }
        \draw [MediumPurple3, very thick] (12, 3.75) -- (6, 3.75);

        \fill [MediumPurple3!15] (6, 3.75) [domain = 2 : 1, samples = 20] plot({24 * (\x - 1)^2/(\x^2)}, {3.5*(\x - 1) + 0.25}) -- (0, 3.75) -- (6, 3.75);
        \draw [DarkSlateGray3, very thick] (12, 3.78) -- (6, 3.78);
        \draw [DarkSlateGray3!87!MediumPurple3, very thick] [domain = 2 : 1, samples = 20] plot({24 * (\x - 1)^2/(\x^2)}, {3.5*(\x - 1) + 0.25});

        \draw [Latex-Latex, gray] (-0.5, 0) -- (12.5, 0) node [gray, anchor = west] {$\frac{\ell}{k^2}$};
        \draw [Latex-Latex, gray] (0, -0.5) -- (0, 4.25) node [gray, anchor = south] {$c$};
        
        \begin{scope}
            \foreach \i\j in {1/2, 1/4, 3/16, 9/64, 27/256} {
                \draw [gray] (24*\i/\j, 0.15) -- (24*\i/\j, -0.15) node [anchor = north] {$\frac{\i}{\j}$};
            }
            \foreach \i in {2, 1} {
                \draw [gray] (0.15, {3.5*(\i - 1) + 0.25}) -- (-0.15, {3.5*(\i - 1) + 0.25}) node [anchor = east] {$\i$};
            }
            \foreach \i\j in {4/3, 8/7} {
                \draw [gray] (0.15, {3.5*(\i/\j - 1) + 0.25}) -- (-0.15, {3.5*(\i/\j - 1) + 0.25}) node [anchor = east] {$\frac{\i}{\j}$};
            }
            \draw [gray, dashed] (0, 0.25) -- (12, 0.25);
            \draw [gray, dashed] (0, 3.75) -- (6, 3.75);
            \draw [gray, dashed] (12, 0) -- (12, 3.75);
        \end{scope}
    \end{tikzpicture}
    \caption{This graph depicts the relationship between the coefficients of $k^2$ in $\ell$ (on the $x$-axis) and the exponents $c$ (on the $y$-axis) for which we can bound $g(n, k, \ell)$ above or below by $n^c$. The blue line shows the upper bounds in Theorem \ref{thm:upper-bound}, and the purple line shows the lower bounds in Theorem \ref{thm:lower-bound}.}
    \label{fig:bounds-graph}
\end{figure}
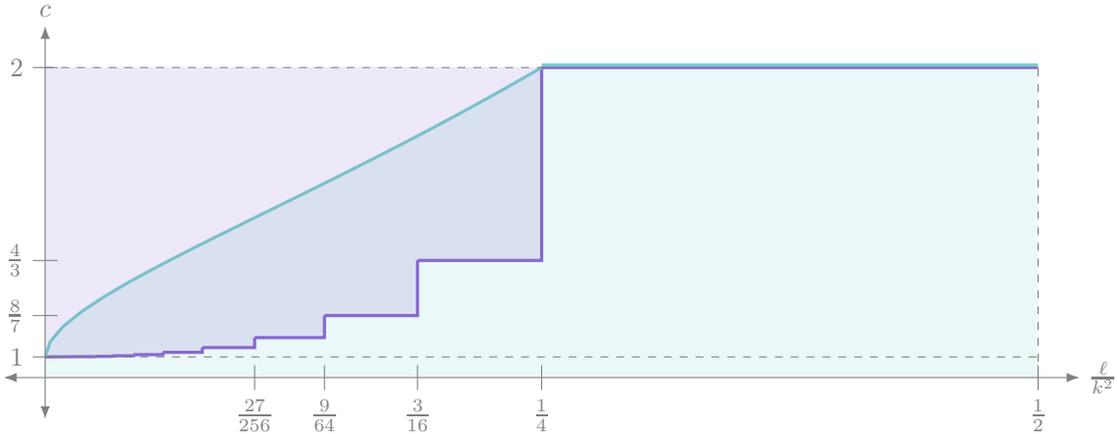

In Section \ref{sec:upper-bound}, we prove Theorem \ref{thm:upper-bound}. In Section \ref{sec:odd-upper-bound}, we prove Proposition \ref{prop:odd-upper-bound}, by strengthening the previous proof in the case where $c = 2$ and $k$ is odd. In Section \ref{sec:lower-bound}, we prove Theorem \ref{thm:lower-bound}. Finally, in Section \ref{sec:consequences}, we deduce Corollaries \ref{cor:even-qt}, \ref{cor:odd-qt}, \ref{cor:nc-threshold}, \ref{cor:poly-bounds}, and \ref{cor:sk-loglogk}.

\section{Proof of Theorem \ref{thm:upper-bound}: A family of upper bounds on \texorpdfstring{$g(n, k, \ell)$}{g(n, k, l)}} \label{sec:upper-bound}

In this section, we prove Theorem \ref{thm:upper-bound} (our family of upper bounds). The structure of the proof is as follows: let $\ell$ be as in Theorem \ref{thm:upper-bound}. In order to show that $g(n, k, \ell) = o(n^c)$, we need to show that for all $a > 0$, for all sufficiently large $n$ we have $g(n, k, \ell) \leq an^c$ --- i.e., there exists an $n$-element set $A \subseteq \RR$ with $\abs{A - A} \leq an^c$ satisfying the $(k, \ell)$-local property. We will obtain such a set $A$ using a random construction. In order to ensure that our set $A$ satisfies the $(k, \ell)$-local property, for every $k$-element subset $A' \subseteq A$, we consider the `configuration' of equal differences formed by its elements. Our argument will then consist of two steps: first, we show that using a random construction, we can obtain an $n$-element set $A$ with $\abs{A - A} \leq an^c$ that avoids all `bad' configurations (where we classify configurations as good or bad based on certain properties). We then show that any $k$-element subset forming a `good' configuration necessarily contains at least $\ell$ distinct differences. It then follows that the set $A$ we constructed must satisfy the $(k, \ell)$-local property, providing the desired upper bound. 

In Subsection \ref{subsec:setup}, we set up a few useful definitions (in particular, we formalize the notion of a configuration of equal differences) and state these two steps more precisely (with definitions of which configurations we consider good and bad). In Subsection \ref{subsec:random-construct}, we describe our random construction and prove that it avoids all bad configurations. In Subsection \ref{subsec:bounding-diffs}, we prove that any subset forming a good configuration contains at least $\ell$ distinct differences, modulo a certain technical lemma; in Subsection \ref{subsec:2s-certify}, we prove this lemma. 

\subsection{Setup} \label{subsec:setup} 

We formalize the notion of a configuration of equal differences in the following way: we define a \emph{$k$-configuration} to be a collection of linear equations over $\RR$ in the $k$ variables $x_1$, \ldots, $x_k$, in which each equation is of the form \begin{align}
    x_{i_1} - x_{i_2} - x_{i_3} + x_{i_4} = 0 \label{eqn:generic-diffeq}
\end{align} 
for some (not necessarily distinct) indices $i_1, i_2, i_3, i_4 \in \{1, \ldots, k\}$ such that $\{i_1, i_4\} \neq \{i_2, i_3\}$. We refer to such an equation as a \emph{difference equality}. (The name, as well as the reason we care about such equations, comes from the fact that \eqref{eqn:generic-diffeq} is a rearrangement of $x_{i_1} - x_{i_2} = x_{i_3} - x_{i_4}$. The condition $\{i_1, i_4\} \neq \{i_2, i_3\}$ means that we only consider equations that are not identically satisfied.) 

In the remainder of this subsection, we establish some useful definitions regarding $k$-configurations; using these definitions, we then state our two main steps more precisely, as Lemmas \ref{lem:random-construct} and \ref{lem:bounding-diffs}. 

\subsubsection{Conventions} 

We first fix a few conventions regarding linear equations (working over the field $\RR$). 

We consider equations to be unique only up to rearrangement; for example, we consider \eqref{eqn:generic-diffeq} to be the same equation as $-x_{i_1} + x_{i_2} + x_{i_3} - x_{i_4} = 0$, $x_{i_1} - x_{i_2} = x_{i_3} - x_{i_4}$, and $x_{i_1} + x_{i_4} = x_{i_2} + x_{i_3}$. 

Given a linear equation $(*)$ in $x_1$, \ldots, $x_k$, we define its \emph{content} to be an expression $*$ (which is a linear combination of $x_1$, \ldots, $x_k$) such that $(*)$ is the equation ${*} = 0$. (For any nontrivial equation $(*)$, there are two ways to define the content of $(*)$, where each is the negative of the other; we choose one arbitrarily.) For example, the content of the difference equality \eqref{eqn:generic-diffeq} is either $x_{i_1} - x_{i_2} - x_{i_3} + x_{i_4}$ or its negative.

We say that a variable $x_i$ \emph{appears} in a linear equation $(*)$, or that $(*)$ \emph{involves} $x_i$, if the coefficient of $x_i$ in $(*)$ is nonzero. (For example, the equation $x_1 - x_2 = 0$ involves $x_1$ and $x_2$, but does not involve $x_3$; this remains the case even if we write it as $x_1 - x_2 - x_3 + x_3 = 0$.)  We also use the same terminology for linear expressions $*$ in place of linear equations $(*)$. Given a set $I \subseteq \{1, \ldots, k\}$, we say $(*)$ is an \emph{equation on $I$} if all variables $x_i$ appearing in $(*)$ satisfy $i \in I$. 

We say linear equations $(*_1)$, \ldots, $(*_t)$ are \emph{linearly independent} if their contents $*_1$, \ldots, $*_t$ are linearly independent (viewed as elements of the $k$-dimensional $\RR$-vector space with basis vectors $x_1$, \ldots, $x_k$). 

Given linear equations $(*_1)$, \ldots, $(*_t)$, $(*)$ with contents $*_1$, \ldots, $*_t$, $*$, we say that the collection $\cT = \{(*_1), \ldots, (*_t)\}$ \emph{implies} $(*)$ if $*$ can be written as a linear combination of $*_1$, \ldots, $*_t$, i.e., \[{*} = \eps_1{*_1} + \cdots + \eps_t{*_t}\] for some $\eps_1, \ldots, \eps_t \in \RR$. (The reason for this terminology is that this is equivalent to the condition that every solution to $\cT$ (viewed as a system of equations) is also a solution to $(*)$; however, we will not need this fact.) We say $\cT$ \emph{minimally implies} $(*)$ if $\cT$ implies $(*)$, but no proper subset $\cT' \subset \cT$ does. (Equivalently, $\cT$ minimally implies $(*)$ if and only if $(*_1)$, \ldots, $(*_t)$ are linearly independent and the coefficients $\eps_1$, \ldots, $\eps_t$ above are all nonzero.) 

We say two collections of linear equations $\cC_1$ and $\cC_2$ are \emph{equivalent} if $\cC_1$ implies every equation in $\cC_2$, and $\cC_2$ implies every equation in $\cC_1$. 

\subsubsection{Definitions for \texorpdfstring{$k$}{k}-configurations}

For a $k$-configuration $\cC$, we let $\Sol(\cC) \subseteq \RR^k$ denote the set of solutions to $\cC$ (where we view $\cC$ as a system of linear equations over $\RR$). For each $I \subseteq \{1, \ldots, k\}$, we let $\Sol_I(\cC)$ denote the projection of $\Sol(\cC)$ onto the coordinates in $I$ --- in other words, we define \[\Sol_I(\cC) = \{(x_i)_{i \in I} \mid (x_1, \ldots, x_k) \in \Sol(\cC)\}.\] We refer to the elements of $\Sol_I(\cC)$ as \emph{$I$-solutions to $\cC$}. 

We define the \emph{dimension} of $\cC$, denoted $\dim(\cC)$, to be the dimension of $\Sol(\cC)$ as a $\RR$-vector space; similarly, for each $I \subseteq \{1, \ldots, k\}$, we define $\dim_I(\cC)$ to be the dimension of $\Sol_I(\cC)$. 

We say a solution to $\cC$ is \emph{generic} if it does not satisfy any difference equalities other than the ones implied by $\cC$. Of course, any solution to $\cC$ must satisfy all equations implied by $\cC$, so a generic solution to $\cC$ satisfies all the difference equalities implied by $\cC$ and no others. (This definition will be used to encapsulate what we mean by a $k$-element set `forming' a configuration of equal differences, as stated in the outline in the beginning of Section \ref{sec:upper-bound}.)

We define an \emph{occurrence} of a $k$-configuration $\cC$ in a set $A \subseteq \RR$ to be a solution to $\cC$ consisting of \emph{distinct} elements of $A$; if there is an occurrence of $\cC$ in $A$, we say $\cC$ \emph{occurs} in $A$. For $I \subseteq \{1, \ldots, k\}$, we define an \emph{$I$-occurrence} of $\cC$ to be an $I$-solution to $\cC$ consisting of distinct elements of $A$. 

\begin{example}
    We illustrate these definitions for the $6$-configuration \[\cC = \{x_1 - x_2 - x_3 + x_4 = 0, \, x_1 - x_2 - x_5 + x_6 = 0\}.\]
    \begin{itemize}
        \item We have $\dim(\cC) = 4$ --- to obtain a solution to $\cC$, we can choose $x_1$, $x_2$, $x_3$, and $x_5$ arbitrarily, and then $x_4$ and $x_6$ are uniquely determined. For the same reason, we have $\dim_{\{1, 2, 3, 5\}}(\cC) = 4$ and $\dim_{\{1, 2, 3, 4\}}(\cC) = 3$ (as $x_4$ is uniquely determined by $x_1$, $x_2$, and $x_3$). Similarly, we have $\dim_{\{3, 4, 5, 6\}}(\cC) = 3$ --- this is because $\cC$ implies $x_3 - x_4 - x_5 + x_6 = 0$, so we can choose $x_3$, $x_4$, and $x_5$ arbitrarily, and then $x_6$ is uniquely determined. 
        \item The $6$-tuple $(1, 2, 4, 5, 9, 10)$ is a solution to $\cC$, but it is not generic, as it satisfies the additional difference equality $x_1 - x_4 = x_4 - x_5$. By contrast, $(1, 2, 4, 5, 10, 11)$ is a generic solution to $\cC$.  
        \item Consider the set $A = \{1, 2, 4, 5, 9, 10, 11\}$. Then $(1, 2, 4, 5, 9, 10)$ and $(1, 2, 4, 5, 10, 11)$ are both occurrences of $\cC$ in $A$. Any $4$-tuple of distinct elements of $A$ is a $\{1, 2, 3, 5\}$-occurrence of $\cC$ in $A$, and $(1, 10, 2, 11)$ is a $\{1, 2, 3, 4\}$-occurrence and a $\{3, 4, 5, 6\}$-occurrence of $\cC$ in $A$. 
    \end{itemize}
\end{example}

We are now ready to define the properties of a $k$-configuration that determine whether we consider them `good' or `bad' (as in the outline of our argument given at the beginning of Section \ref{sec:upper-bound}). 

First, we say that a $k$-configuration is \emph{invalid} if it implies $x_{i_1} - x_{i_2} = 0$ for some $i_1 \neq i_2$ (i.e., it implies two variables are equal), and \emph{valid} otherwise. Note that an invalid $k$-configuration can never occur in a set $A$ (because of the distinctness condition in the definition of an occurrence). Also note that in a generic solution to a valid $k$-configuration, all $k$ entries must be distinct. 

We say a $k$-configuration is \emph{AP-containing} if it implies $x_{i_1} - 2x_{i_2} + x_{i_3} = 0$ for some distinct $i_1$, $i_2$, $i_3$ (i.e., it implies three variables form an arithmetic progression), and \emph{AP-free} otherwise. 

Most importantly, given any $1 < c \leq 2$, we say a $k$-configuration $\cC$ is \emph{$c$-heavy} if there exists some $1 \leq m \leq k$ and some $m$-element subset $I \subseteq \{1, \ldots, k\}$ such that \[\dim_I(\cC) < \frac{(c - 1)m + 1}{c},\] and we call such a subset $I$ a \emph{$c$-heavy part} of $\cC$; we say $\cC$ is \emph{$c$-light} if it is not $c$-heavy. (This definition was chosen in order to make the expected value calculation in the proof of Claim \ref{claim:heavy-occs} work out.)

\begin{remark}\label{rmk:c-light-alt}
    Note that if $\cC$ is $c$-light, then for any $I \subseteq \{1, \ldots, k\}$ of size $\abs{I} = m$, any collection of linearly independent equations on $I$ implied by $\cC$ has size at most \[m - \dim_I(\cC) \leq m - \frac{(c - 1)m + 1}{c} = \frac{m - 1}{c}.\] Equivalently, any collection of $t$ linearly independent equations implied by $\cC$ must involve at least $ct + 1$ variables. We will use this observation frequently in our proof of Proposition \ref{prop:odd-upper-bound} (with $c = 2$). 
\end{remark}

Finally, we say a $k$-configuration $\cC$ is \emph{$c$-good} if it is valid, AP-free, and $c$-light; we say it is \emph{$c$-bad} otherwise. Note that if a $k$-configuration $\cC$ has any of these three properties (i.e., is valid, AP-free, or $c$-light), then any collection of difference equalities implied by $\cC$ has the same property. 

\begin{example}\label{ex:2-good-illustrate}
    To illustrate these definitions, we consider a few examples with $c = 2$. 
    \begin{enumerate}[(a)]
        \item The $4$-configuration $\{x_1 - x_2 - x_3 + x_4 = 0, \, x_1 - x_2 + x_3 - x_4 = 0\}$ is invalid, because it implies $x_3 = x_4$. (It is also $2$-heavy, with both $\{3, 4\}$ and $\{1, 2, 3, 4\}$ as $2$-heavy parts.) 
        \item The $5$-configuration $\{x_1 - x_2 - x_3 + x_4 = 0, \, x_1 + x_2 - x_3 - x_5 = 0\}$ is valid and $2$-light, but it is AP-containing (as it implies $x_5 - 2x_2 + x_4 = 0$) and therefore $2$-bad. 
        \item The $12$-configuration $\{x_1 - x_2 - x_3 + x_4 = 0, \, x_1 - x_2 - x_5 + x_6 = 0, \, x_1 - x_2 - x_7 + x_8 = 0, \, x_1 - x_3 - x_5 + x_7 = 0, \, x_9 - x_{10} - x_{11} + x_{12} = 0\}$ is valid and AP-free, but it is $2$-heavy with $2$-heavy part $\{1, \ldots, 8\}$ (as it contains $4$ linearly independent equations on $\{1, \ldots, 8\}$), and therefore $2$-bad. 
        \item \label{item:ex-2-good} The $6$-configuration $\cC = \{x_1 - x_2 - x_5 + x_4 = 0, \, x_1 - x_3 - x_6 + x_4 = 0\}$ --- equivalently, $\{x_1 + x_4 = x_2 + x_5 = x_3 + x_6\}$ --- is $2$-good, as it is valid, AP-free, and $2$-light. The fact that it is $2$-light can be seen by pairing $x_1$ and $x_4$, $x_2$ and $x_5$, and $x_3$ and $x_6$ --- we can obtain the solutions to $\cC$ by freely choosing the values of both variables in one pair and one variable in each of the other pairs (then the values of the remaining two variables are uniquely determined), and for any $I \subseteq \{1, \ldots, 6\}$ of size $m \geq 1$, this gives a way of freely choosing the values of more than half of the variables $x_i$ with indices $i \in I$, thus showing that $\dim_I(\cC) \geq \frac{m + 1}{2}$. 
    \end{enumerate}
\end{example}

\subsubsection{Our main steps}

Finally, we state the results of the two main steps of our argument as lemmas and explain why they imply Theorem \ref{thm:upper-bound}. 

The following lemma will be the result of the first step of our argument, where we use a random construction to obtain a set $A$ with few distinct differences that `avoids' all $c$-bad $k$-configurations.

\begin{lemma} \label{lem:random-construct}
    Fix a positive integer $k$ and real numbers $1 < c \leq 2$ and $a > 0$. Then for all sufficiently large $n$ (with respect to $k$, $c$, and $a$), there exists an $n$-element set $A \subseteq \RR$ with $\abs{A - A} \leq an^c$ such that every $k$-configuration that occurs in $A$ is $c$-good. 
\end{lemma}

The following lemma will be the result of the second step of our argument, in which we show that a generic solution to a $c$-good $k$-configuration contains many distinct differences. In fact, we show more generally that a generic solution to any valid $k$-configuration with high dimension contains many distinct differences. 

\begin{lemma} \label{lem:bounding-diffs}
    Fix a positive integer $k$, let $\cC$ be a valid $k$-configuration, and let $\dim(\cC) = d$. Let $(a_1, \ldots, a_k)$ be a generic solution to $\cC$, and let $A' = \{a_1, \ldots, a_k\}$. Then 
    \begin{align} 
        \abs{A' - A'} \geq \begin{cases} (d - 1)^2 & \text{if $d \leq \frac{k}{2} + 1$} \\ \binom{k}{2} - (k - d)(k - d + 1) & \text{if $d \geq \frac{k}{2} + 1$}.\end{cases} \label{eqn:bounding-diffs-bound}
    \end{align}
\end{lemma}

Note that the two expressions in \eqref{eqn:bounding-diffs-bound} are equal at $d = \frac{k}{2} + 1$, and both are (weakly) increasing in $d$; this means that the entire right-hand side of \eqref{eqn:bounding-diffs-bound} is increasing in $d$. 

We will prove these two lemmas in Subsections \ref{subsec:random-construct} and \ref{subsec:bounding-diffs}, respectively. For now, we explain why they together imply Theorem \ref{thm:upper-bound}. 

\begin{proof}[Proof of Theorem \ref{thm:upper-bound}]
    By the comments at the beginning of Section \ref{sec:upper-bound}, it suffices to show that the sets $A$ given by Lemma \ref{lem:random-construct} satisfy the $(k, \ell)$-local property for the value of $\ell$ in Theorem \ref{thm:upper-bound}. To see this, suppose that a set $A$ has the property that every $k$-configuration occurring in $A$ is $c$-good. Then for any $k$-element subset $A' = \{a_1, \ldots, a_k\} \subseteq A$, we can write down the $k$-configuration $\cC$ consisting of \emph{all} the difference equalities satisfied by $(a_1, \ldots, a_k)$. Then $\cC$ occurs in $A$, so it must be $c$-good; in particular, $\cC$ is valid and \[\dim(\cC) \geq \ceil{\frac{(c - 1)k + 1}{c}}\] (by the $c$-light condition applied to the entire set $\{1, \ldots, k\}$, and the fact that $\dim(\cC)$ is an integer). Meanwhile $(a_1, \ldots, a_k)$ is a generic solution to $\cC$ (as every difference equality it satisfies is in fact \emph{contained} in $\cC$), so we can bound $\abs{A' - A'}$ using Lemma \ref{lem:bounding-diffs}. The bound given by \eqref{eqn:bounding-diffs-bound} is increasing in $d$ and we have \[\ceil{\frac{(c - 1)k + 1}{c}} \leq \ceil{\frac{k + 1}{2}} \leq \frac{k}{2} + 1,\] so Lemma \ref{lem:bounding-diffs} gives that \[\abs{A' - A'} \geq \left(\ceil{\frac{(c - 1)k + 1}{c}} - 1\right)^2 = \ceil{\frac{(c - 1)(k - 1)}{c}}^2 = \ell.\] This means $A$ indeed satisfies the $(k, \ell)$-local property, as desired. 
\end{proof}

\begin{remark}
    For this argument, we do not need the full strength of Lemma \ref{lem:random-construct} --- the only parts of the $c$-good condition that we need are that a $c$-good configuration $\cC$ is valid and satisfies $\dim(\cC) \geq \frac{(c - 1)k + 1}{c}$. However, we will need the remaining parts of the $c$-good condition in our proof of Proposition \ref{prop:odd-upper-bound} (and we state Lemma \ref{lem:random-construct} in this way so that we can use it for that proof as well).
\end{remark}

\subsection{Proof of Lemma \ref{lem:random-construct}} \label{subsec:random-construct}

In this subsection, we prove Lemma \ref{lem:random-construct} using a random construction (with the alteration method). Our construction will make use of the following theorem of Behrend. 

\begin{theorem}[Behrend \cite{Beh46}]\label{thm:behrend}
    For every positive integer $n$, there exists a set $S \subseteq \{1, \ldots, n\}$ of size $\abs{S} = n^{1 - o(1)}$ containing no $3$-term arithmetic progression. 
\end{theorem}

\begin{proof}[Proof of Lemma \ref{lem:random-construct}]
    Assume that $n$ is sufficiently large with respect to $k$, $c$, and $a$ (all our asymptotic notation treats $k$, $c$, and $a$ as fixed, and the implied constants will depend on them). First, by Theorem \ref{thm:behrend}, there exists a set $S \subseteq \{1, \ldots, \ceil{an^c}\}$ of size $s = n^{c - o(1)}$ with no $3$-term arithmetic progression. We will ultimately take $A$ to be an $n$-element subset of $S$, which will automatically imply that $\abs{A - A} \leq an^c$ (as we have $A - A \subseteq S - S \subseteq \{1, \ldots, \ceil{an^c} - 1\}$) and that every $k$-configuration that occurs in $A$ is AP-free. So in order to prove Lemma \ref{lem:random-construct}, it will suffice to ensure that every $k$-configuration that occurs in $A$ is $c$-light. (We automatically have that every $k$-configuration that occurs in $A$ is valid, as this is true for \emph{any} set.)

    Let $p = \frac{3n}{s}$. Note that $p = n^{1 - c + o(1)}$, so $p \in (0, 1)$ for all sufficiently large $n$. Let $T$ be a randomly chosen subset of $S$ where each element is included independently with probability $p$. Then $\EE[\abs{T}] = 3n$, and since $\abs{T}$ is a sum of $s$ independent Bernoulli random variables, by the Chernoff bounds we have $\abs{T} \geq 2n$ with probability $1 - o(1)$. 

    We will show that with probability $1 - o(1)$, it is possible to destroy all occurrences of $c$-heavy $k$-configurations in $T$ by deleting at most $n$ of its elements (which we will then do to obtain our set $A$). The following claim allows us to do so. 

    \begin{claim}\label{claim:heavy-occs}
        Let $\cC$ be a $c$-heavy $k$-configuration, and let $I$ be a $c$-heavy part of $\cC$. Then the expected number of $I$-occurrences of $\cC$ in $T$ is $o(n)$. 
    \end{claim}

    \begin{proof}
        Let $d = \dim_I(\cC)$. First, we claim that there are at most $s^d$ $I$-occurrences of $\cC$ in $S$. To see this, without loss of generality assume $I = \{1, \ldots, m\}$ for some $m$. By solving $\cC$ using reduced row echelon form, we can find a list of indices $i_1 < \cdots < i_{d'}$, where $d' = \dim(\cC)$, such that every choice of values for $x_{i_1}$, \ldots, $x_{i_{d'}}$ in $\RR$ corresponds to a unique solution $(x_1, \ldots, x_k)$ to $\cC$, and for each $i \in \{1, \ldots, k\}$, the value of $x_i$ in this solution depends on only the chosen values of $x_{i_j}$ for the indices $i_j < i$. Then each choice of values for $x_{i_j}$ over all indices $i_j \leq m$ corresponds to a unique $I$-solution to $\cC$, so there must be exactly $\dim_I(\cC) = d$ such indices. Finally, each choice of values for these $d$ variables $x_{i_j}$ in $S$ provides at most one $I$-solution to $\cC$ whose elements are in $S$, and therefore at most one $I$-occurrence of $\cC$ in $S$. There are $s^d$ such choices, so there are at most $s^d$ $I$-occurrences of $\cC$ in $S$. 

        Now for each $I$-occurrence of $\cC$ in $S$, the probability that all its elements are included in $T$ is $p^m$ (as it has $m$ distinct elements, each of which is included in $T$ independently with probability $p$). So by the linearity of expectation, the expected number of $I$-occurrences of $\cC$ in $T$ is at most $s^d \cdot p^m$. Since $s = n^{c - o(1)}$ and $p = n^{1 - c + o(1)}$, we have \[s^d \cdot p^m = n^{cd + m(1 - c) + o(1)}.\] Since $cd + m(1 - c) < 1$ by the definition of a $c$-heavy part (and this quantity is a constant not depending on $n$), we have $s^d \cdot p^m = o(n)$, as desired. 
    \end{proof}

    Now for each $c$-heavy $k$-configuration $\cC$, fix some $c$-heavy part of $\cC$, which we denote by $h(\cC)$. Then Claim \ref{claim:heavy-occs} implies that for each \emph{fixed} $c$-heavy $\cC$, the number of $h(\cC)$-occurrences of $\cC$ in $T$ has expected value $o(n)$. But the total number of $k$-configurations is a constant (depending on $k$ but not $n$), since for any given $k$, there are only finitely many difference equalities. So the \emph{total} number of $h(\cC)$-occurrences of $\cC$ over all $c$-heavy $\cC$ has expected value $o(n)$ as well. Then by Markov's inequality, this number is at most $n$ with probability $1 - o(1)$. 

    So with probability $1 - o(1)$, we have $\abs{T} \geq 2n$ and there are at most $n$ total $h(\cC)$-occurrences of $\cC$ in $T$ over all $c$-heavy $k$-configurations $\cC$; this means that for sufficiently large $n$, there must exist a set $T \subseteq S$ with both properties. Fix such a set $T$, and let $A$ be an $n$-element subset of $T$ obtained by first deleting some element from each $h(\cC)$-occurrences of each $c$-heavy $\cC$ (which requires us to delete at most $n$ elements), and then deleting additional elements arbitrarily until we are left with exactly $n$ elements. 

    We have already shown that $\abs{A - A} \leq an^c$ and that no invalid or AP-containing $k$-configuration can occur in $A$; this construction guarantees that no $c$-heavy $k$-configuration can occur in $A$ (as any occurrence of a $c$-heavy $k$-configuration $\cC$ would contain a $h(\cC)$-occurrence of $\cC$). So $A$ has all the claimed properties. 
\end{proof}

\subsection{Proof of Lemma \ref{lem:bounding-diffs}} \label{subsec:bounding-diffs}

In this subsection, we prove Lemma \ref{lem:bounding-diffs} (our lower bound on the number of distinct differences in a generic solution to a valid $k$-configuration, in terms of the dimension of the $k$-configuration). We first establish a few definitions that will be useful in this proof. 

Note that equations of the form $x_{i_1} - x_{i_2} = 0$, $2x_{i_1} - 2x_{i_2} = 0$, and $x_{i_1} - 2x_{i_2} + x_{i_3} = 0$ are all considered to be difference equalities, since we do not require the indices in \eqref{eqn:generic-diffeq} to be distinct. We say a difference equality is \emph{degenerate} if it is of the form $x_{i_1} - x_{i_2} = 0$ for distinct indices $i_1$ and $i_2$ (equivalently, in \eqref{eqn:generic-diffeq} the sets $\{i_1, i_4\}$ and $\{i_2, i_3\}$ are different but not disjoint), and \emph{nondegenerate} otherwise. (Since we assume the $k$-configuration $\cC$ we work with is valid, all difference equalities it implies will be nondegenerate.)

Given a nondegenerate difference equality $(*)$, we say that two variables $x_i$ and $x_j$ appear in $(*)$ \emph{with opposite sign} if $(*)$ can be written as $x_i - x_j = x_{i'} - x_{j'}$ for some $i', j' \in \{1, \ldots, k\}$, and \emph{with the same sign} if $(*)$ can be written as $x_i + x_j = x_{i'} + x_{j'}$ for some $i', j' \in \{1, \ldots, k\}$. This definition matches the intuitive meaning of these terms; in the difference equality \eqref{eqn:generic-diffeq}, $x_{i_1}$ appears with opposite sign as both $x_{i_2}$ and $x_{i_3}$, and with the same sign as $x_{i_4}$. 

Given distinct indices $i, j \in \{1, \ldots, k\}$, we say that a collection of linear equations \emph{$i$-certifies $j$} if it implies some nondegenerate difference equality in which $x_i$ and $x_j$ appear with opposite sign. (In particular, we say a single difference equality $i$-certifies $j$ if $x_i$ and $x_j$ appear in it with opposite sign.) 

Our proof of Lemma \ref{lem:bounding-diffs} will require the following lemma. The proof of this lemma is somewhat technical, so we defer it to Subsection \ref{subsec:2s-certify} (and first prove Lemma \ref{lem:bounding-diffs} assuming this lemma). 

\begin{lemma}\label{lem:2s-certify}
    Fix $i \in \{1, \ldots, k\}$, and let $\cS$ be a valid collection of linearly independent difference equalities that all involve $x_i$. Then $\cS$ $i$-certifies at most $2\abs{\cS}$ indices. 
\end{lemma}

\begin{proof}[Proof of Lemma \ref{lem:bounding-diffs}]
    We first make a convenient assumption on the ordering of the indices of $x_1$, \ldots, $x_k$: by using reduced row echelon form to solve $\cC$ (as a system of equations), we can find $d$ distinct indices $i_1$, \ldots, $i_d$ such that every choice of values for $x_{i_1}$, \ldots, $x_{i_d}$ gives a unique solution $(x_1, \ldots, x_k)$ to $\cC$. We assume without loss of generality that these indices are $1$, \ldots, $d$. 

    We can calculate $\abs{A' - A'}$ in the following way: for each pair of indices $(i, j)$ with $1 \leq j < i \leq k$, we say $(i, j)$ corresponds to a \emph{repeated difference} if $\abs{a_i - a_j} = \abs{a_{i'} - a_{j'}}$ for some $(i', j')$ with $1 \leq j' < i' \leq k$ that is lexicographically smaller than $(i, j)$ (i.e., either $i' < i$, or $i' = i$ and $j' < j$), and a \emph{new difference} otherwise. Note that if a pair $(i, j)$ corresponds to a repeated difference, then $(a_1, \ldots, a_k)$ satisfies a nondegenerate difference equality of the form 
    \begin{align}
        x_i - x_j - x_{i'} + x_{j'} = 0\label{eqn:rep-diff}
    \end{align} 
    for some $i', j' \leq i$ (here we do not necessarily have $j' < i'$) --- i.e., a difference equality on $\{1, \ldots, i\}$ in which $x_i$ and $x_j$ occur with opposite sign. Since $(a_1, \ldots, a_k)$ is a generic solution to $\cC$, this difference equality \eqref{eqn:rep-diff} must be implied by $\cC$.
    
    Then $\abs{A' - A'}$ is precisely the number of pairs $(i, j)$ corresponding to new differences. For each index $i$, we will upper-bound the number of repeated differences contributed by $i$ (i.e., the number of indices $j < i$ for which $(i, j)$ corresponds to a repeated difference) by bounding the number of indices $j < i$ for which $\cC$ can imply a difference equality of the form \eqref{eqn:rep-diff}. Combining these bounds over all $i$ will then give the desired lower bound on $\abs{A' - A'}$. 
    
    We first consider the indices $i \leq d$. Our assumption that every choice of $x_1$, \ldots, $x_d$ corresponds to a solution to $\cC$ means that $\cC$ cannot imply any equation on $\{1, \ldots, d\}$, and therefore cannot imply a difference equality of the form \eqref{eqn:rep-diff} for any $j < i \leq d$. So each index $i \leq d$ contributes no repeated differences.

    We now consider the indices $i > d$. Fix some index $i > d$, let $\ol{\cS_i}$ be the set of all difference equalities on $\{1, \ldots, i\}$ involving $x_i$ that are implied by $\cC$, and let $\cS_i$ be a maximal linearly independent subset of $\ol{\cS_i}$. Then $\cS_i$ implies every equation in $\ol{\cS_i}$, so if a pair $(i, j)$ corresponds to a repeated difference, then $\cS_i$ must $i$-certify $j$ (as the corresponding difference equality \eqref{eqn:rep-diff} must be in $\ol{\cS_i}$, and is therefore implied by $\cS_i$). 

    But we have $\dim_{\{1, \ldots, i\}}(\cC) = d$, since each choice of $x_1$, \ldots, $x_d$ corresponds to a unique solution to $\cC$. This means we must have $\abs{\cS_i} \leq i - d$ (using the fact from linear algebra that a system of linear equations on $i$ variables consisting of $s$ linearly independent equations has a solution space of dimension $i - s$). Then by Lemma \ref{lem:2s-certify}, $\cS_i$ $i$-certifies at most $2(i - d)$ indices, and therefore $i$ contributes at most $2(i - d)$ repeated differences. 

    We now use these statements to lower-bound $\abs{A' - A'}$. Each index $i \leq d$ contributes no repeated differences, and therefore $i - 1$ new differences; meanwhile, each index $i > d$ contributes at most $2(i - d)$ repeated differences, and therefore at least $i - 1 - 2(i - d) = 2d - i - 1$ new differences. So $\abs{A' - A'}$ --- i.e., the total number of pairs corresponding to new differences --- is at least \[\sum_{i = 1}^d (i - 1) + \sum_{i = d + 1}^k \max(2d - i - 1, 0) = \begin{cases} (d - 1)^2 & \text{if $d \leq \frac{k}{2} + 1$} \\ \binom{k}{2} - (k - d)(k - d + 1) & \text{if $d \geq \frac{k}{2} + 1$}.\end{cases}\] (For $d \geq \frac{k}{2} + 1$, it is easier to calculate this bound by noting that there are at most $2(1 + 2 + \cdots + (k - d)) = (k - d)(k - d + 1)$ pairs corresponding to repeated differences, and therefore at least $\binom{k}{2} - (k - d)(k - d + 1)$ pairs corresponding to new differences.) 
\end{proof}

\begin{remark}\label{rmk:bounding-diffs-equality}
    Lemma \ref{lem:bounding-diffs} is tight for $d \geq \frac{k}{2} + 1$, as equality is achieved by the $k$-configuration \[\cC = \{x_1 + x_d = x_2 + x_{d + 1} = \cdots = x_{k - d + 1} + x_k\}.\] (The fact that this $k$-configuration is $2$-good can be seen by the same argument as in Example \ref{ex:2-good-illustrate}\ref{item:ex-2-good}.) Each index $1 \leq i \leq d$ contributes no repeated differences, and each index $i > d$ contributes exactly $2(i - d)$ repeated differences, corresponding to the pairs $(i, j)$ and $(i, j - d + 1)$ for all $d \leq j < i$ (as $\cC$ implies $x_i + x_{i - d + 1} - x_j - x_{j - d + 1} = 0$ for all such $j$, so in any generic solution we have $a_i - a_j = a_{j - d + 1} - a_{i - d + 1}$ and $a_i - a_{j - d + 1} = a_j - a_{i - d + 1}$). So there are exactly $(k - d)(k - d + 1)$ pairs corresponding to repeated differences. 
\end{remark}

\subsection{Proof of Lemma \ref{lem:2s-certify}} \label{subsec:2s-certify}

In this subsection, we finally prove Lemma \ref{lem:2s-certify} (thus completing the proof of Theorem \ref{thm:upper-bound}). Note that if the collection $\cS$ in Lemma \ref{lem:2s-certify} does not imply any difference equalities involving $x_i$ other than the ones already in $\cS$, then the conclusion is obvious, as each difference equality in $\cS$ individually $i$-certifies at most $2$ indices. The difficulty comes from the fact that $\cS$ may imply more difference equalities involving $x_i$, which may $i$-certify additional indices; the key observation that allows us to deal with this is the following lemma. 

\begin{lemma}\label{lem:gen-min-impl}
    Let $\cT = \{(*_1), \ldots, (*_t)\}$ be a valid collection of $t$ linearly independent difference equalities, and suppose that $\cT$ minimally implies a difference equality $(*_\cT)$. Then the following statements hold: 
    \begin{enumerate}[(i)]
        \item \label{item:gmi-few-total} At most $2t + 2$ variables appear in $\cT$. Furthermore, if equality holds, then every variable appearing in $\cT$ appears in exactly two of $(*_1)$, \ldots, $(*_t)$, $(*_\cT)$. 
        \item \label{item:gmi-few-new} For every proper subset $\cT' \subseteq \cT$ of size $t'$, at most $2t'$ variables appear in $\cT'$ but not $\cT \setminus \cT'$. 
    \end{enumerate}
\end{lemma}

\begin{proof}
    Let $*_1$, \ldots, $*_t$, $*_\cT$ be the contents of $(*_1)$, \ldots, $(*_t)$, $(*_\cT)$. Since $\cT$ minimally implies $(*_\cT)$, there exist nonzero $\eps_1, \ldots, \eps_t \in \RR$ for which 
    \begin{align}
        {*} = \eps_1{*_1} + \cdots + \eps_t{*_t}.\label{eqn:gen-min-impl}
    \end{align}

    We first prove \ref{item:gmi-few-total}. Note that \eqref{eqn:gen-min-impl} implies that variable that appears in $\cT$ must appear in at least two of $(*_1)$, \ldots, $(*_t)$, $(*_\cT)$. But each of these $t + 1$ equations involves at most $4$ variables, so there are at most $4(t + 1)$ appearances of variables with multiplicity. So the total number of variables appearing in $\cT$ is at most $\frac{4(t + 1)}{2} = 2t + 2$, and if equality holds then each variable must appear exactly twice. 

    We now prove \ref{item:gmi-few-new}. For notational convenience, assume that $\cT' = \{(*_1), \ldots, (*_{t'})\}$, and let 
    \begin{align}
        {*_{\cT \setminus \cT'}} = \eps_{t' + 1}{*_{t' + 1}} + \cdots + \eps_t{*_t}.\label{eqn:gmi-old}
    \end{align} 
    (In words, $*_{\cT \setminus \cT'}$ is the partial sum of \eqref{eqn:gen-min-impl} corresponding to only the equations \emph{not} in $\cT'$.) Then we have
    \begin{align}
        {*_\cT} - {*_{\cT \setminus \cT'}} = \eps_1{*_1} + \cdots + \eps_t{*_t}.\label{eqn:gmi-new-old}
    \end{align}
    Let $(*_{\cT \setminus \cT'})$ be the equation ${*_{\cT \setminus \cT'}} = 0$ (this is implied by $\cT$, but is not necessarily a difference equality).  

    We now define three disjoint sets of indices:
    \begin{itemize}
        \item Let $I_1$ be the set of indices $i$ for which $x_i$ appears in $(*_{\cT \setminus \cT'})$ but not $(*_\cT)$. 
        \item Let $I_2$ be the set of indices for which $x_i$ appears in $(*_\cT)$ but not $(*_{\cT \setminus \cT'})$. 
        \item Let $I_3$ be the set of indices $i$ for which $x_i$ appears in $\cT'$, but not $(*_{\cT \setminus \cT'})$ or $(*_\cT)$. 
    \end{itemize} 
    Let $I_1$, $I_2$, and $I_3$ have sizes $m_1$, $m_2$, and $m_3$, respectively. 

    Then for every variable $x_i$ appearing in $\cT'$ but not $\cT \setminus \cT'$, we must have $i \in I_2$ or $i \in I_3$ --- if $x_i$ does not appear in $\cT \setminus \cT'$, then by \eqref{eqn:gmi-old} it certainly does not apply in $(*_{\cT \setminus \cT'})$, so $i$ is in $I_2$ if $x_i$ appears in $(*_\cT)$ and $I_3$ if it does not. This means the number of variables appearing in $\cT'$ but not $\cT \setminus \cT'$ is at most $m_2 + m_3$, so it suffices to show that $m_2 + m_3 \leq 2t'$. 

    First, we claim that $m_1 +  m_2 + 2m_3 \leq 4t'$. To see this, for each $i \in I_1 \cup I_2$, since $x_i$ appears on the left-hand side of \eqref{eqn:gmi-new-old}, it must also appear on the right-hand side; this means $x_i$ appears in at least one equation in $\cT'$. Meanwhile, for each $i \in I_3$, since $x_i$ does not appear on the left-hand side of \eqref{eqn:gmi-new-old} but does appear in at least one summand on the right-hand side, it must appear in at least two; this means $x_i$ appears in at least \emph{two} equations in $\cT'$. Since each of the $t'$ equations in $\cT'$ involves at most $4$ variables, this means that $m_1 + m_2 + 2m_3 \leq 4t'$, as claimed. 

    On the other hand, we will show that $m_1 \geq m_2 - 1$. First, $(*_\cT)$ involves at most $4$ variables, as it is a difference equality. Meanwhile, we claim $(*_{\cT \setminus \cT'})$ involves at least $3$ variables. To see this, first note that $*_{\cT \setminus \cT'}$ is not simply $0$ (as $\cT \setminus \cT'$ is nonempty and its equations are linearly independent). However, its coefficients must sum to $0$, as this is true of every summand in \eqref{eqn:gen-min-impl}. So $(*_{\cT \setminus \cT'})$ cannot involve only $1$ variable, and if it involved only $2$ variables, then it would be of the form $\eps x_i - \eps x_j = 0$ for some $\eps \neq 0$ and $i \neq j$, contradicting the validity of $\cT$. So $(*_{\cT \setminus \cT'})$ must involve at least $3$ variables. Finally, $m_1 - m_2$ is the difference between the numbers of variables involved in $(*_{\cT \setminus \cT'})$ and $(*_\cT)$, so $m_1 - m_2 \geq 3 - 4 = -1$. 

    Combining these two bounds, we have $2(m_2 + m_3) - 1 \leq m_1 + m_2 + 2m_3 \leq 4t'$. Since $2(m_2 + m_3)$ and $4t'$ are both even, this means $2(m_2 + m_3) \leq 4t'$, so $m_2 + m_3 \leq 2t'$, as desired. 
\end{proof}

We are now ready to prove Lemma \ref{lem:2s-certify}. 

\begin{proof}[Proof of Lemma \ref{lem:2s-certify}]
    Write down all difference equalities in $\cS$. For each subset of $\cS$ that minimally implies some difference equality also involving $x_i$, draw a box around this subset. We say the \emph{size} of a box is the number of equations it contains, and we say a box is \emph{big} if it has size greater than $1$ (each equation in $\cS$ is automatically in a box of size $1$, as it minimally implies itself). An example is given in Figure \ref{fig:boxes-config}. 
    
    \begin{figure}[ht]
        \centering
        \begin{tikzpicture}

            \filldraw [fill = Honeydew3!20, draw = Honeydew3] (3.8, -4.3) rectangle (8.2, -1.3);
            \filldraw [fill = DarkSlateGray3!10, draw = DarkSlateGray3] (-2, -2.5) rectangle (2, 0.5);
            \fill [MediumPurple3!10] (-2.2, -4.5) rectangle (2.2, -1.5);
            \fill [MediumPurple3!50!DarkSlateGray3!20] (-2, -2.5) rectangle (2, -1.5);
            \draw [DarkSlateGray3!91!MediumPurple3] (-2, -1.5) -- (-2, -2.5) -- (2, -2.5) -- (2, -1.5);
            \draw [MediumPurple3] (-2.2, -4.5) rectangle (2.2, -1.5);

            \node [draw] at (0, 0) {$x_1 - x_2 - x_3 + x_4 = 0$};
            \node [draw] at (0, -1) {$x_1 + x_2 - x_5 - x_6 = 0$};
            \node [draw] at (0, -2) {$x_1 - x_3 - x_5 + x_7 = 0$};
            
            \node [DarkSlateGray3, anchor = south] at (0, 0.5) {$(*_1) : x_1 + x_4 - x_6 - x_7 = 0$};
            \node [draw] at (0, -3) {$x_1 + x_3 - x_8 - x_9 = 0$};
            \node [draw] at (0, -4) {$x_1 - x_5 - x_8 + x_{10} = 0$};
            
            \node [MediumPurple3, anchor = north] at (0, -4.5) {$(*_2) : x_1 + x_7 - x_9 - x_{10} = 0$};
            \node [draw] at (6, -0.2) {$x_1 - x_4 - x_{11} + x_{12} = 0$};
            \begin{scope}[yshift = -0.3cm]
                \node [draw] at (6, -1.5) {$x_1 - x_{12} - x_{13} + x_{14} = 0$};
                \node [draw] at (6, -2.5) {$x_1 + x_{13} - x_{15} - x_{16} = 0$};
                \node [draw] at (6, -3.5) {$x_1 + x_{14} - x_{15} - x_{17} = 0$};
                \node [Honeydew3!50!Honeydew4, anchor = north] at (6, -4) {$(*_3) : x_1 - x_{12} - x_{16} + x_{17} = 0$};
            \end{scope}
        \end{tikzpicture}
        \caption{This figure shows an example of a possible $\cS$ and the corresponding boxes. Here $i = 1$, and $\cS$ consists of the equations written in black. Each equation of $\cS$ is in a box of size $1$, and we have three boxes of size $3$, which are drawn in blue, purple, and green and labelled with the difference equalities that they minimally imply (named $(*_1)$, $(*_2)$, and $(*_3)$, respectively). For example, the blue box minimally implies $(*_1)$ because $(x_1 - x_2 - x_3 + x_4) + (x_1 + x_2 - x_5 - x_6) - (x_1 - x_3 - x_5 + x_7) = x_1 + x_4 - x_6 - x_7$.}
        \label{fig:boxes-config}
    \end{figure}
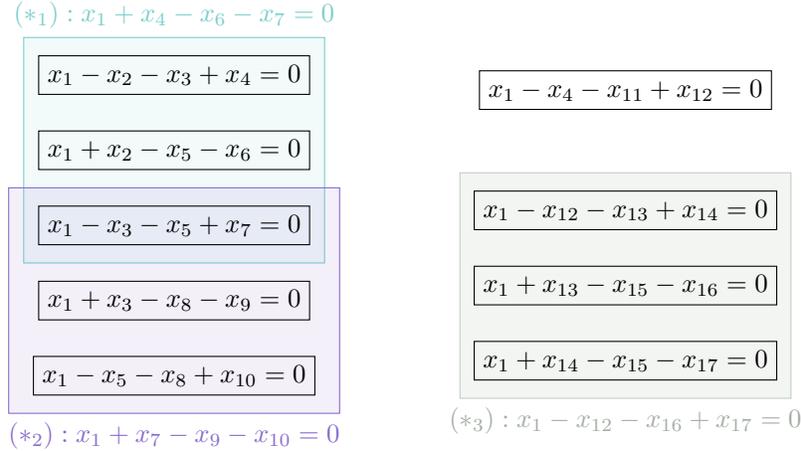

    We will work with the `connected components' of this picture, which we refer to as `blobs' --- to formalize this, consider the graph on $\cS$ in which two equations of $\cS$ are adjacent if and only if there is some box containing both of them. We say a \emph{blob} is a connected component of this graph, and the \emph{size} of a blob is again the number of equations it contains. In the example in Figure \ref{fig:boxes-config}, we have one blob of size $5$ (the union of the blue and purple boxes), one blob of size $3$ (the green box), and one blob of size $1$ (the one equation not in any big box). 

    Then the blobs form a partition of $\cS$ such that if $\cS$ $i$-certifies an index $j$, then in fact some blob must $i$-certify $j$ (every difference equality involving $x_i$ that is implied by $\cS$ is by definition implied by some box, and therefore by the blob containing this box). We will show that for each blob, the number of indices it $i$-certifies is at most twice its size.

    Clearly any blob of size $1$ can $i$-certify at most $2$ indices. Meanwhile, we will show that any blob of size $t > 1$ can only \emph{involve} at most $2t$ variables other than $x_i$, and therefore can $i$-certify at most $2t$ indices. 

    Any blob $\cT$ of size $t > 1$ can be written as a union of big boxes $\cT_0$, \ldots, $\cT_r$ such that for each $1 \leq p \leq r$, the box $\cT_p$ has nonempty intersection with $\cT_0 \cup \cdots \cup \cT_{p - 1}$. Let $t_0 = \abs{\cT_0}$, and for each $1 \leq p \leq r$, let $t_p = \abs{\cT_p \setminus (\cT_0 \cup \cdots \cup \cT_{p - 1})}$; then $t = t_0 + \cdots + t_r$. (Intuitively, we can imagine `building' $\cT$ by starting with some big box $\cT_0$, and then repeatedly adding in big boxes $\cT_1$, \ldots, $\cT_r$ such that at each step, the next big box we add overlaps with some previously added box; then for each $p$, we define $t_p$ as the number of new equations we add to the growing blob when we add in $\cT_p$. In the example in Figure \ref{fig:boxes-config}, for the blob of size $5$, we can take $\cT_0$ to be the blue box and $\cT_1$ to be the purple box; then $t_0 = 3$ and $t_1 = 2$.)

    First, Lemma \ref{lem:gen-min-impl}\ref{item:gmi-few-total} implies that $\cT_0$ contains at most $2t_0 + 2$ variables. Furthermore, equality cannot hold, as $x_i$ appears in all $t_0 + 1 \geq 3$ of the relevant equations, so in fact $\cT_0$ contains at most $2t_0 + 1$ variables. 

    Meanwhile, for each $1 \leq p \leq r$, Lemma \ref{lem:gen-min-impl}\ref{item:gmi-few-new} applied to $\cT_p \setminus (\cT_0 \cup \cdots \cup \cT_{p - 1})$ (which is a proper subset of $\cT_p$) gives that there are at most $2t_p$ variables that appear in $\cT_p$ but not $\cT_0 \cup \cdots \cup \cT_{p - 1}$. (Rephrased in terms of the process of building $\cT$ one box at a time, this states that when we add in the big box $\cT_p$ --- thus adding in $t_p$ new equations to our growing blob --- these $t_p$ new equations give us at most $2t_p$ new variables.) 
    
    So the total number of variables that appear in $\cT$ is at most \[(2t_0 + 1) + 2t_1 + \cdots + 2t_p = 2t + 1.\] Since one of these variables is $x_i$, this means at most $2t$ variables other than $x_i$ appear in $\cT$, as desired. 

    So for every blob, the number of indices it $i$-certifies is at most twice its size. Since the blobs form a partition of $\cS$, this means the \emph{total} number of indices $i$-certified by any blob is at most $2\abs{\cS}$, as desired. 
\end{proof}

Thus we have shown Lemma \ref{lem:2s-certify}, which completes the proof of Theorem \ref{thm:upper-bound} (our family of upper bounds on $g(n, k, \ell)$).  

\section{Proof of Proposition \ref{prop:odd-upper-bound}: A stronger subquadratic bound for odd \texorpdfstring{$k$}{k}} \label{sec:odd-upper-bound}

In this section, we prove Proposition \ref{prop:odd-upper-bound} (our improvement on the value of $\ell$ in Theorem \ref{thm:upper-bound} when $c = 2$ and $k$ is odd). We use the same random construction as in our proof of Theorem \ref{thm:upper-bound} (i.e., we use Lemma \ref{lem:random-construct} as written), but improve our analysis of the number of distinct differences in a generic solution to a $2$-good $k$-configuration when $k$ is odd --- we replace Lemma \ref{lem:bounding-diffs} with the following bound. 

\begin{lemma}\label{lem:odd-bounding-diffs}
    Fix an odd positive integer $k$, and let $\cC$ be a $2$-good $k$-configuration. Let $(a_1, \ldots, a_k)$ be a generic solution to $\cC$, and let $A' = \{a_1, \ldots, a_k\}$. Then \[\abs{A' - A'} \geq \frac{(k + 1)^2}{4} - 4.\] 
\end{lemma}

For comparison, Lemma \ref{lem:bounding-diffs} gives only that $\abs{A' - A'} \geq \frac{(k - 1)^2}{4}$.

Together, Lemmas \ref{lem:random-construct} and \ref{lem:odd-bounding-diffs} immediately imply Proposition \ref{prop:odd-upper-bound} for the same reason that Lemmas \ref{lem:random-construct} and \ref{lem:bounding-diffs} imply Theorem \ref{thm:upper-bound} --- Lemma \ref{lem:odd-bounding-diffs} implies that the sets $A$ given by Lemma \ref{lem:random-construct} (when $c = 2$ and $k$ is odd) satisfy the $(k, \ell)$-local property for the value of $\ell$ in Proposition \ref{prop:odd-upper-bound}.  

Similarly to our proof of Lemma \ref{lem:bounding-diffs}, our proof of Lemma \ref{lem:odd-bounding-diffs} will rely on two lemmas whose proofs are rather technical. In Subsection \ref{subsec:odd-bounding-diffs}, we state these lemmas and prove Lemma \ref{lem:odd-bounding-diffs} assuming them. In Subsection \ref{subsec:observe}, we prove a few observations that will be useful in the proofs of both lemmas. Finally, in Subsection \ref{subsec:2s-certify-equality} we prove the first technical lemma, and in Subsection \ref{subsec:impls-not-i} we prove the second. 

\subsection{Proof of Lemma \ref{lem:odd-bounding-diffs}} \label{subsec:odd-bounding-diffs} 

Our proof of Lemma \ref{lem:odd-bounding-diffs} will make use of the following definition (in addition to the ones given at the beginning of Subsection \ref{subsec:bounding-diffs}): we say that two distinct nondegenerate difference equalities $(*_1)$ and $(*_2)$ are \emph{difference-aligned at $(i, j)$} for distinct indices $i, j \in \{1, \ldots, k\}$ if the variables $x_i$ and $x_j$ appear with opposite sign in both $(*_1)$ and $(*_2)$, and \emph{sum-aligned at $(i, j)$} if $x_i$ and $x_j$ appear with the same sign in both $(*_1)$ and $(*_2)$. For example, $x_1 - x_2 - x_3 + x_4 = 0$ and $x_1 - x_2 - x_5 + x_6 = 0$ are difference-aligned at $(1, 2)$, while $x_1 - x_2 - x_3 + x_4 = 0$ and $x_1 + x_4 - x_5 - x_6 = 0$ are sum-aligned at $(1, 4)$. We say $(*_1)$ and $(*_2)$ are \emph{difference-aligned} or \emph{sum-aligned} if there exist indices $i$ and $j$ for which they are difference-aligned or sum-aligned at $(i, j)$.

Note that if a $k$-configuration $\cC$ is valid and AP-free, then all difference equalities it implies must involve exactly $4$ distinct variables (as a difference equality \emph{not} involving $4$ distinct variables is of one of the forms $x_{i_1} - x_{i_2} = 0$, $2x_{i_1} - 2x_{i_2} = 0$, or $x_{i_1} - 2x_{i_2} + x_{i_3} = 0$; the first two contradict the fact that $\cC$ is valid, and the third contradicts the fact that $\cC$ is AP-free.)

We first make a simple observation that will make these definitions more convenient to work with. (This observation will also be useful in several other ways for the proofs of Lemma \ref{lem:2s-certify-equality} and \ref{lem:impls-not-i}.)

\begin{lemma}\label{lem:2-eqns-5-vars}
    Let $(*_1)$ and $(*_2)$ be distinct difference equalities such that $\{(*_1), (*_2)\}$ is valid and AP-free. Then $(*_1)$ and $(*_2)$ share at most $2$ variables. 
\end{lemma}

\begin{proof}
    Assume for contradiction that $(*_1)$ and $(*_2)$ share at least $3$ variables, so in total they involve at most $5$ variables. Let $*_1$ and $*_2$ be the contents of $(*_1)$ and $(*_2)$. We can assume without loss of generality that at least $2$ of the shared variables have the same coefficient in both $*_1$ and $*_2$ (either at least $2$ have the same coefficient, or at least $2$ have opposite coefficients, and in the latter case we can replace $*_2$ by its negative to reach the former case). 

    Now let ${*} = {*_1} - {*_2}$, and let $(*)$ be the equation ${*} = 0$ (which is implied by $\{(*_1), (*_2)\}$). The variables that appear with the same coefficient in $*_1$ and $*_2$ must cancel out of $*$, so $(*)$ involves at most $5 - 2 = 3$ variables. Furthermore, $*$ is not identically $0$ (as $(*_1)$ and $(*_2)$ are distinct), but its coefficients sum to $0$ and are all $\pm 1$ or $\pm 2$. So $(*)$ must be of the form $x_{i_1} - x_{i_2} = 0$, $2x_{i_1} - 2x_{i_2} = 0$, or $x_{i_1} - 2x_{i_2} + x_{i_3} = 0$, contradicting the fact that $\{(*_1), (*_2)\}$ is valid and AP-free. 
\end{proof}

In particular, Lemma \ref{lem:2-eqns-5-vars} means that if $(*_1)$ and $(*_2)$ share some variable $x_i$ and are difference-aligned or sum-aligned, then they must be difference-aligned or sum-aligned at $(i, j)$ for some index $j \neq i$. In our argument, we will frequently work with collections of difference equalities that all involve some fixed variable $x_i$; this observation will allow us to simply write that two equations in such a collection are difference-aligned or sum-aligned when we mean that they are difference-aligned or sum-aligned at $(i, j)$ for some $j$. 

We now state the two technical lemmas we need for our proof of Lemma \ref{lem:odd-bounding-diffs}. The first lemma gives a characterization of the $2$-good collections $\cS$ that achieve equality or near-equality in Lemma \ref{lem:2s-certify}. (As stated, the lemma only gives a necessary condition on $\cS$; it is not hard to show that the converse of the lemma is true as well --- i.e., this condition is sufficient --- but we do not need this for our purposes.) 

\begin{lemma}\label{lem:2s-certify-equality}
    Fix $i \in \{1, \ldots, k\}$, and let $\cS$ be a $2$-good collection consisting of $s$ linearly independent difference equalities that all involve $x_i$, such that $\cS$ $i$-certifies at least $2s - 1$ indices. Then we can find some $\cS'$, also consisting of $s$ linearly independent difference equalities that all involve $x_i$, such that $\cS'$ is equivalent to $\cS$ and the following statements hold: 
    \begin{enumerate}[(i)]
        \item \label{item:certify-2s} If $\cS$ $i$-certifies exactly $2s$ indices, then no two equations in $\cS'$ are difference-aligned. 
        \item \label{item:certify-one-less} If $\cS$ $i$-certifies exactly $2s - 1$ indices, then exactly one pair of equations in $\cS'$ is difference-aligned. 
    \end{enumerate}
\end{lemma}

The second technical lemma essentially considers a collection of difference equalities of the form produced by Lemma \ref{lem:2s-certify-equality}, and bounds the number of difference equalities \emph{not} involving $x_i$ that such a collection implies. 

\begin{lemma}\label{lem:impls-not-i}
    Fix $i \in \{1, \ldots, k\}$, and suppose that $\cS$ is a $2$-good collection of $s$ linearly independent difference equalities that all involve $x_i$, such that at most one pair of equations in $\cS$ is difference-aligned.
    \begin{enumerate}[(i)]
        \item \label{item:no-diff} If no two equations in $\cS$ are difference-aligned and not all are sum-aligned, then $\cS$ implies at most $\binom{s - 1}{2}$ difference equalities that do not involve $x_i$. 
        \item \label{item:one-diff} If exactly one pair of equations in $\cS$ is difference-aligned, then $\cS$ implies at most $\binom{s - 1}{2} + 1$ difference equalities that do not involve $x_i$. 
        \item \label{item:all-sum} If all equations in $\cS$ are sum-aligned, then $\cS$ implies at most $\binom{s}{2}$ difference equalities that do not involve $x_i$.
    \end{enumerate}
\end{lemma}

Note that the cases \ref{item:one-diff} and \ref{item:all-sum} are disjoint (and of course \ref{item:no-diff} is also disjoint from both of them) --- by Lemma \ref{lem:2-eqns-5-vars}, $\cS$ cannot contain two equations that are simultaneously difference-aligned and sum-aligned. In fact, equality always holds in \ref{item:all-sum}, but we do not need this for our proof. 

We now use these lemmas to prove Lemma \ref{lem:odd-bounding-diffs}. 

\begin{proof}[Proof of Lemma \ref{lem:odd-bounding-diffs}]
    Let $\dim(\cC) = d$; since $\cC$ is $2$-good, we must have $d \geq \frac{k + 1}{2}$. First note that if $d \geq \frac{k + 3}{2}$, then we are done by Lemma \ref{lem:bounding-diffs}, as we have \[\abs{A' - A'} \geq \binom{k}{2} - \frac{(k - 3)(k - 1)}{4} = \frac{(k + 1)^2}{4} - 1 > \frac{(k + 1)^2}{4} - 4.\] So we now assume that $d = \frac{k + 1}{2}$.

    In order to lower-bound $\abs{A' - A'}$, we use the same notion of new and repeated differences as in our proof of Lemma \ref{lem:bounding-diffs} --- we will show that the number of pairs corresponding to repeated differences is at most \[(k - d)(k - d - 1) + 3 = \frac{(k - 3)(k - 1)}{4} + 3,\] and therefore the number of pairs corresponding to new differences is at least \[\binom{k}{2} - \frac{(k - 3)(k - 1)}{4} - 3 = \frac{(k + 1)^2}{4} - 4.\] (For comparison, in Lemma \ref{lem:bounding-diffs} we proved the number of pairs corresponding to repeated differences was at most $(k - d)(k - d + 1)$; so we wish to improve this bound by $2(k - d) - 3$.) As before, if a pair $(i, j)$ with $i > j$ corresponds to a repeated difference, then $\cC$ must imply some nondegenerate difference equality 
    \begin{align}
        x_i - x_j - x_{i'} + x_{j'} = 0 \label{eqn:rep-diff-2}
    \end{align}
    with $i', j' < i$, i.e., a nondegenerate difference equality on $\{1, \ldots, i\}$ in which $x_i$ and $x_j$ appear with opposite sign. Similarly to in the proof of Lemma \ref{lem:bounding-diffs}, we will first make a certain assumption (without loss of generality) on the ordering of the indices of $x_1$, \ldots, $x_k$ (although this assumption will be different from the one we used to prove Lemma \ref{lem:bounding-diffs}), and then bound the number of pairs $(i, j)$ for which $\cC$ implies a difference equality of the form \eqref{eqn:rep-diff-2}.

    Our assumption on the ordering of indices will make use of the following definition. Given some set $I \subseteq \{1, \ldots, k\}$ and index $i \in I$, we say that $I$ is \emph{$i$-saturated} if it satisfies both  of the following conditions: 
    \begin{itemize}
        \item We have $\dim_I(\cC) = d$.
        \item There are at least $2(\abs{I} - d) - 1$ indices $j \in I$ for which $\cC$ implies a difference equality on $I$ in which $x_i$ and $x_j$ appear with opposite sign.
    \end{itemize}
    We say that $I$ is \emph{saturated} if it is $i$-saturated for some $i \in I$. 

    Note that there must exist some saturated set of indices --- by again using reduced row echelon form to solve $\cC$, we can find $d$ distinct indices $i_1$, \ldots, $i_d$ such that every choice of $x_{i_1}$, \ldots, $x_{i_d}$ corresponds to a unique solution $(x_1, \ldots, x_k)$ to $\cC$, and then the set $\{i_1, \ldots, i_d\}$ is saturated (as the second condition is trivially satisfied). So we can let $I \subseteq \{1, \ldots, k\}$ be a \emph{maximal} saturated set (meaning that $I$ is saturated, but no proper superset of $I$ is). Without loss of generality, assume that $I = \{1, \ldots, i_0\}$ for some $d \leq i_0 \leq k$, and that $I$ is $i_0$-saturated --- this is the assumption we make on the ordering of the variables. 
    
    Then the structure of our argument is as follows: 
    \begin{enumerate}[(1)]
        \item \label{item:post-sat} We will individually upper-bound the number of repeated differences contributed by each $i > i_0$ using the fact that $\{1, \ldots, i\}$ is not $i$-saturated. (This bound will be $2$ less than the bound we would get using the same argument as in our proof of Lemma \ref{lem:bounding-diffs}.)
        \item \label{item:sat} We will bound the number of repeated differences contributed by $i = i_0$ in the same way as in our proof of Lemma \ref{lem:bounding-diffs}. 
        \item \label{item:pre-sat} Finally, we will bound the \emph{total} number of repeated differences contributed by all $i < i_0$ by combining the fact that $I = \{1, \ldots, i_0\}$ is $i_0$-saturated together with Lemmas \ref{lem:2s-certify-equality} and \ref{lem:impls-not-i}. The bound we get from this will `usually' be significantly less than the bound we would get using a similar argument to the one in our proof of Lemma \ref{lem:bounding-diffs}; then combining it with \ref{item:post-sat} and \ref{item:sat} gives the desired bound. 
        
        \item \label{item:next-index} The one case in which this does \emph{not} happen (i.e., the bound in step \ref{item:pre-sat} is not good enough) is when we are in case \ref{item:all-sum} when we apply Lemma \ref{lem:impls-not-i}. When this happens, roughly speaking, the difference equalities $\cC$ implies on $\{1, \ldots, i_0\}$ have the same structure as in the equality case of Lemma \ref{lem:bounding-diffs} described in Remark \ref{rmk:bounding-diffs-equality}, so it is not possible to get a better bound on the number of repeated differences contributed by $i \leq i_0$. However, using this structure, we will show that the index $i_0 + 1$ exists and contributes much fewer repeated differences than our original bound in \ref{item:post-sat}; this adjustment to \ref{item:post-sat} then gives the desired bound. 
    \end{enumerate}

    We now proceed with the argument. For step \ref{item:post-sat}, for each $i > i_0$, the fact that $\{1, \ldots, i\}$ is not $i$-saturated means that there are strictly fewer than $2(i - d) - 1$ indices for which $\cC$ implies a difference equality of the form \eqref{eqn:rep-diff-2}, and therefore $i$ contributes at most $2(i - d - 1)$ repeated differences. 

    Now, to perform steps \ref{item:sat} and \ref{item:pre-sat}, let $\ol{\cS}$ be the set of difference equalities on $I = \{1, \ldots, i_0\}$ involving $x_{i_0}$ that are implied by $\cC$, and let $\cS$ be a maximal linearly independent subset of $\ol{\cS}$. As in the proof of Lemma \ref{lem:bounding-diffs}, since $\dim_I(\cC) = d$, we must have $\abs{\cS} \leq i_0 - d$. On the other hand, since $I$ is $i_0$-saturated, $\cS$ must $i_0$-certify at least $2(i_0 - d) - 1$ indices, so by Lemma \ref{lem:2s-certify} we must have $\abs{\cS} \geq i_0 - d$. So we must have $\abs{\cS} = i_0 - d$, and $\cS$ must $i_0$-certify at least $2\abs{\cS} - 1$ indices. Then we can apply Lemma \ref{lem:2s-certify-equality} to $\cS$; let $\cS'$ be the collection of difference equalities that it provides, so that at most one pair of equations in $\cS'$ is difference-aligned. Since $\cS$ and $\cS'$ are equivalent (and each is a collection of linearly independent equations), we have $\abs{\cS'} = \abs{\cS} = i_0 - d$.  

    For step \ref{item:sat}, the number of repeated differences contributed by the index $i = i_0$ is at most the number of indices $i_0$-certified by $\cS$, which is either $2(i_0 - d)$ or $2(i_0 - d) - 1$. 
    
    Meanwhile, for step \ref{item:pre-sat}, the total number of pairs $(i, j)$ with $i < i_0$ corresponding to repeated differences is at most twice the number of difference equalities on $\{1, \ldots, i_0 - 1\}$ implied by $\cC$ --- the equation \eqref{eqn:rep-diff-2} corresponding to each such pair is a difference equality on $\{1, \ldots, i_0 - 1\}$ implied by $\cC$, and each such difference equality corresponds to at most $2$ such pairs $(i, j)$ (as $i$ must be the largest index present in the difference equality, and $j$ must be one of the two indices for which $x_i$ and $x_j$ appear with opposite sign). 
    
    But as stated earlier, since $\dim_I(\cC) = d$, any linearly independent collection of equations on $I$ implied by $\cC$ has size at most $i_0 - d$. Since $\cS'$ already has size $i_0 - d$, this means every equation on $I$ implied by $\cC$ is in fact implied by $\cS'$. So the number of difference equalities on $\{1, \ldots, i_0 - 1\}$ implied by $\cC$ is precisely the number of difference equalities implied by $\cS'$ not involving $x_{i_0}$, which we can upper-bound using Lemma \ref{lem:impls-not-i} (since at most one pair of equations in $\cS'$ is difference-aligned).

    We now consider the following three cases (based on which case of Lemma \ref{lem:impls-not-i} applies to $\cS'$).

    \begin{enumerate}[(a)]
        \item First suppose that no equations in $\cS'$ are difference-aligned and not all are sum-aligned. Then $\cS$ must fall under the first case of Lemma \ref{lem:2s-certify-equality} (since $\cS'$ has no difference-aligned pairs if $\cS$ falls under the first case, and one difference-aligned pair if $\cS$ falls under the second), so it $i_0$-certifies exactly $2(i_0 - d)$ indices. Meanwhile, by Lemma \ref{lem:impls-not-i}\ref{item:no-diff}, $\cS'$ implies at most $\binom{i_0 - d - 1}{2}$ difference equalities not involving $x_{i_0}$. 
        
        Now combining our bounds from the steps \ref{item:post-sat}, \ref{item:sat}, and \ref{item:pre-sat}, we have that the total number of pairs $(i, j)$ corresponding to repeated differences (over all indices $1 \leq i \leq k$) is at most \[2\binom{i_0 - d - 1}{2} + 2(i_0 - d) + \sum_{i = i_0 + 1}^k 2(i - d - 1) = (k - d)(k - d - 1) + 2.\] (One intuitive way to see why this calculation works is to compare this to the proof of Lemma \ref{lem:bounding-diffs}. There, we used the bound that each index $i > d$ contributes at most $2(i - d)$ repeated differences, giving a bound of $2(1 + 2 + \cdots + (k - d)) = (k - d)(k - d + 1)$ on the total number of repeated differences. Here we save $2$ for each index $i > i_0$; and for the indices $i < i_0$ we have a \emph{total} contribution of $2\binom{i_0 - d - 1}{2}$ instead of $2(1 + 2 + \cdots + (i_0 - d - 1)) = 2\binom{i_0 - d}{2}$, which means we save $2(i_0 - d - 1)$. So in total, we save $2(k - i_0) + 2(i_0 - d - 1) = 2(k - d) - 2$ compared to the bound in Lemma \ref{lem:bounding-diffs}.)
        
        \item Now suppose that exactly one pair of equations in $\cS'$ is difference-aligned. Then $\cS$ must fall under the second case of Lemma \ref{lem:2s-certify-equality}, so it $i_0$-certifies exactly $2(i_0 - d) - 1$ indices. Meanwhile, by Lemma \ref{lem:impls-not-i}\ref{item:one-diff}, $\cS'$ implies at most $\binom{i_0 - d - 1}{2} + 1$ difference equalities not involving $x_{i_0}$. Again combining our three bounds, the total number of pairs corresponding to repeated differences is at most \[2\binom{i_0 - d - 1}{2} + 2 + (2(i_0 - d) - 1) + \sum_{i = i_0 + 1}^k 2(i - d - 1) = (k - d)(k - d - 1) + 3.\] 
        \item Finally, suppose that all equations in $\cS'$ are sum-aligned. This means none are difference-aligned, so $\cS$ again falls under the first case of Lemma \ref{lem:2s-certify-equality}, and therefore $i_0$-certifies exactly $2(i_0 - d)$ indices. Meanwhile, by Lemma \ref{lem:impls-not-i}\ref{item:all-sum}, $\cS'$ implies at most $\binom{i_0 - d}{2}$ difference equalities not involving $x_{i_0}$.  

        We now perform step \ref{item:next-index}. Since $\cS'$ is valid and AP-free, and all its equations are sum-aligned, there must exist some $j_0 < i_0$ for which all its equations in $\cS'$ are of the form 
        \begin{align}
            x_{i_0} + x_{j_0} = x_i + x_j\label{eqn:sum-aligned}
        \end{align} 
        for distinct $i, j < i_0$, and the $i_0 - d$ pairs $\{i, j\}$ over all these equations must be disjoint from each other and from $\{i_0, j_0\}$. 
        
        In particular, $\cS'$ involves exactly $2(i_0 - d + 1)$ indices, all of which are at most $i_0$. This means we must have $2(i_0 - d + 1) \leq i_0$, so $i_0 \leq 2d - 2 = k - 1$. So $i_0 + 1$ is indeed a valid index.  

        We now show that the index $i = i_0 + 1$ contributes at most $2$ repeated differences. (This is a significant improvement on the bound of $2(i_0 - d)$ from step \ref{item:post-sat} used in the previous two cases, and will make up for the fact that our bound from step \ref{item:pre-sat} is significantly larger in this case.) 
        
        Let $I' = \{1, \ldots, i_0 + 1\}$. If the index $i_0 + 1$ contributes no repeated differences, then there is nothing to prove; now assume that it contributes at least one repeated difference, so there exist distinct $i_1, i_2, i_3, i_4 \leq i_0$ such that $\cC$ implies \begin{align}
            x_{i_0 + 1} - x_{i_1} - x_{i_2} + x_{i_3} = 0.\label{eqn:next-index-repdiff}
        \end{align}
        
        \begin{claim}
            Neither $x_{i_1}$ nor $x_{i_2}$ appears in $\cS'$.
        \end{claim} 

        \begin{proof}
            We will show that if $x_{i_1}$ appears in $\cS'$, then $I'$ is $i_1$-saturated, contradicting the fact that $I$ is a \emph{maximal} saturated set. (Then by symmetry $x_{i_2}$ cannot appear in $\cS'$ either.)

            Assume for contradiction that $x_{i_1}$ appears in $\cS'$, and let \[x_{i_0} + x_{j_0} = x_{i_1} + x_{j_1}\] be the equation it appears in. Then $\cS'$ also implies the difference equalities \[x_{i_1} + x_{j_1} = x_i + x_j\] for each of the $i_0 - d - 1$ pairs $\{i, j\}$ in \eqref{eqn:sum-aligned} except $\{i_1, j_1\}$. Then since $x_{i_0 + 1}$ appears with opposite sign as $x_{i_1}$ in \eqref{eqn:next-index-repdiff}, we have at least $2(i_0 - d) + 1 = 2(i_0 - d + 1) - 1$ variables that occur with opposite sign as $x_{i_1}$ in some difference equality on $I'$ implied by $\cC$ --- namely, all variables appearing in $\cS'$ other than $x_{i_1}$ itself and $x_{j_1}$, together with $x_{i_0 + 1}$. So $I'$ is $i_1$-saturated, which contradicts the maximality of $I$. 
        \end{proof}
        
        Then we claim that $\cC$ cannot imply any difference equality on $I'$ involving $x_{i_0 + 1}$ other than \eqref{eqn:next-index-repdiff}. To see this, assume that $\cC$ implies such a difference equality $(*)$. First, $(*)$ must be implied by $\cS'$ together with \eqref{eqn:next-index-repdiff}, since $\cS'$ and \eqref{eqn:next-index-repdiff} together form a collection of $i_0 - d + 1$ linearly independent equations on $I'$ (and $\dim_{I'}(\cC) = d$, so \emph{any} collection of linearly independent equations on $I'$ implied by $\cC$ has size at most $i_0 - d + 1$). But the linear combination used to obtain $(*)$ must use \eqref{eqn:next-index-repdiff} in order for the variable $x_{i_0 + 1}$ to appear; then since $x_{i_1}$ and $x_{i_2}$ appear in \eqref{eqn:next-index-repdiff} but not in $\cS'$, they must appear in $(*)$. This means \eqref{eqn:next-index-repdiff} and $(*)$ share the three variables $x_{i_0 + 1}$, $x_{i_1}$, and $x_{i_2}$, contradicting Lemma \ref{lem:2-eqns-5-vars}.

        So the index $i_0 + 1$ only contributes $2$ repeated differences (corresponding to the pairs $(i_0 + 1, i_1)$ and $(i_0 + 1, i_2)$), as desired. 

        Then combining steps \ref{item:post-sat}, \ref{item:sat}, \ref{item:pre-sat}, and \ref{item:next-index}, the total number of pairs corresponding to repeated differences is at most \[2\binom{i_0 - d}{2} + 2(i_0 - d) + 2 + \sum_{i = i_0 + 2}^k 2(i - d - 1) = (k - d)(k - d - 1) + 2.\] (This calculation can again be understood by a comparison to Lemma \ref{lem:bounding-diffs} --- here we have the same bounds for the total contribution of all indices $i \leq i_0$, we save $2(i_0 - d)$ in the bound for the contribution of $i_0 + 1$, and we save $2$ for each index $i > i_0 + 1$, for a total saving of $2(i_0 - d) + 2(k - i_0 - 1) = 2(k - d) - 2$.)
    \end{enumerate}

    So in all cases, at most $(k - d)(k - d - 1) + 3$ pairs correspond to repeated differences, as desired. 
\end{proof}

\begin{remark}\label{rmk:odd-bounding-equality}
    Lemma \ref{lem:odd-bounding-diffs} is tight for $k \geq 7$ --- letting $k = 2d - 1$, the $k$-configuration \[\cC = \{x_1 + x_d = x_2 + x_{d + 1} = \cdots = x_{d - 1} + x_{2d - 2}, \, x_1 - x_2 - x_3 + x_{2d - 1} = 0\}\] achieves equality. The pairs corresponding to repeated differences are precisely $(i, j)$ and $(i, j - d + 1)$ for $d \leq j < i \leq 2d - 2$, as well as $(2d - 1, 2)$, $(2d - 1, 3)$, and $(2d - 1, d)$ --- this last pair corresponds to a repeated difference because the equations $x_1 + x_d - x_2 - x_{d + 1} = 0$ and $x_1 - x_2 - x_3 + x_{2d - 1} = 0$ together imply $x_3 + x_d - x_{d + 1} - x_{2d + 1} = 0$. So there are exactly \[2\binom{d - 1}{2} + 3 = \frac{(k - 1)(k - 3)}{4} + 3\] repeated differences. (Note that this $k$-configuration does not satisfy the same assumption on ordering that we make in the proof of Lemma \ref{lem:odd-bounding-diffs} --- here the maximal saturated set is $\{1, \ldots, 2d - 1\}$, which is $2$-saturated and $3$-saturated (and $d$-saturated) --- so we would have to reorder the indices in order to run that proof.)
\end{remark}

\subsection{Some useful observations} \label{subsec:observe}

The proofs of Lemmas \ref{lem:2s-certify-equality} and \ref{lem:impls-not-i} will both involve analyzing subsets of the $2$-good collection $\cS$ that minimally imply certain other difference equalities; the following two lemmas will be useful for this analysis.

\begin{lemma}\label{lem:2-good-min-impl}
    Let $\cT = \{(*_1), \ldots, (*_t)\}$ be a $2$-good collection of $t$ linearly independent difference equalities, and suppose that $\cT$ minimally implies some difference equality $(*_\cT)$. Then the following two statements hold: 
    \begin{enumerate}[(i)]
        \item \label{item:2-good-mi-vars} Either $\cT$ involves exactly $2t + 2$ variables and each of these variables appears exactly twice among $(*_1)$, \ldots, $(*_t)$, $(*_\cT)$; or $\cT$ contains exactly $2t + 1$ variables, one appears exactly four times, and the others all appear exactly twice. 
        \item \label{item:2-good-signs} Let $*_1$, \ldots, $*_t$, $*_\cT$ be the contents of $(*_1)$, \ldots, $(*_t)$, $(*_\cT)$. Then there exist $\eps_1, \ldots, \eps_t \in \{\pm 1\}$ such that we have \[{*_\cT} = \eps_1{*_1} + \cdots + \eps_t{*_t}.\] 
    \end{enumerate}
\end{lemma}

\begin{proof}
    We will first prove a weaker version of \ref{item:2-good-mi-vars} (with one additional possibility); then we will use this weaker version to prove \ref{item:2-good-signs}; finally, we will use \ref{item:2-good-signs} to eliminate the additional possibility and prove \ref{item:2-good-mi-vars}. 

    Let $m$ be the number of variables appearing in $\cT$, and let $e_1$, \ldots, $e_m$ be the numbers of equations among $(*_1)$, \ldots, $(*_t)$, $(*_\cT)$ that each variable appears in. 
    
    On one hand, $\cT$ consists of $t$ linearly independent equations on $m$ variables; so since $\cT$ is $2$-light, by Remark \ref{rmk:c-light-alt} we must have $m \geq 2t + 1$. 

    On the other hand, we can use the same idea as in the proof of Lemma \ref{lem:gen-min-impl} --- first, each equation involves exactly $4$ variables (and every variable appearing in $(*_\cT)$ also appears in $\cT$), so \[e_1 + \cdots + e_m = 4(t + 1).\] Meanwhile, since $\cT$ minimally implies $(*_\cT)$, there must exist nonzero $\eps_1, \ldots, \eps_t \in \RR$ such that 
    \begin{align}
        {*_\cT} = \eps_1{*_1} + \cdots + \eps_t{*_t}, \label{eqn:2-good-impl}
    \end{align}
    which means each of $e_1$, \ldots, $e_m$ must be at least $2$ --- any variable appearing in only one of $(*_1)$, \ldots, $(*_t)$, $(*_\cT)$ would not cancel out of \eqref{eqn:2-good-impl}. This immediately gives $m \leq 2t + 2$, so $m$ must be $2t + 1$ or $2t + 2$. Furthermore, if $m = 2t + 2$, then $e_1$, \ldots, $e_m$ must all be $2$; if $m = 2t + 1$, then either one of $e_1$, \ldots, $e_m$ is $4$ and the rest are $2$, or two are $3$ and the rest are $2$. In order to prove \ref{item:2-good-mi-vars}, it suffices to eliminate the last possibility; we will do so after proving \ref{item:2-good-signs}.

    We now prove \ref{item:2-good-signs}, i.e., that in \eqref{eqn:2-good-impl} the coefficients $\eps_1$, \ldots, $\eps_t$ must all be $\pm 1$. Note that we have already shown that all but at most $2$ variables appearing among $(*_1)$, \ldots, $(*_t)$, $(*_\cT)$ appear exactly twice; we will use this in our proof of \ref{item:2-good-signs}. 
    
    Consider the graph $\cG$ with vertex set $(*_1)$, \ldots, $(*_t)$, $(*_\cT)$, where two equations are adjacent if some variable appears in both of these equations and no others. If two equations are adjacent in $\cG$, then their coefficients in \eqref{eqn:2-good-impl} must have equal magnitude --- otherwise the shared variable would not cancel out of \eqref{eqn:2-good-impl}. So in order to prove \ref{item:2-good-signs}, it suffices to show that $\cG$ is connected. 

    Assume for contradiction that $\cG$ is not connected. Then it must have some connected component not containing $(*_\cT)$; let $\cT' \subseteq \cT$ be the set of equations this connected component consists of. For notational convenience, assume that $\cT' = \{(*_1), \ldots, (*_{t'})\}$ for some $t' \leq t$; then let  
    \begin{align}
        {*_{\cT'}} = \eps_1{*_1} + \cdots + \eps_{t'}{*_{t'}}.\label{eqn:tprime}
    \end{align}
    (In words, $*_{\cT'}$ is the partial sum of \eqref{eqn:2-good-impl} corresponding to only the equations in $\cT'$.) Then we also have 
    \begin{align}
        {*_{\cT'}} = {*_\cT} - \eps_{t' + 1}{*_{t' + 1}} - \cdots - \eps_t{*_t}.\label{eqn:tprime-alt}
    \end{align}
    Let $(*_{\cT'})$ be the equation ${*_{\cT'}} = 0$ (which is implied by $\cT$). 

    We claim that no variable that appears in exactly two of $(*_1)$, \ldots, $(*_t)$, $(*_\cT)$ can appear in $(*_{\cT'})$. For any such variable $x_i$, since $\cT'$ is a connected component of $\cG$, it must contain both or neither of the two equations that $x_i$ appears in. If $\cT'$ contains both, then $x_i$ does not appear in any term on the right-hand side of \eqref{eqn:tprime-alt}; if $\cT'$ contains neither, then $x_i$ does not appear in any term on the right-hand side of \eqref{eqn:tprime}. So in either case, $x_i$ does not appear in $(*_\cT')$. 

    This means at most two variables can appear in $(*_{\cT'})$ (namely, the ones that appear in more than two equations). But $*_{\cT'}$ cannot be $0$ (as the equations in $\cT'$ are linearly independent), and its coefficients must sum to $0$; this means it must be of the form $\eps x_{i_1} - \eps x_{i_2} = 0$ for some $\eps \neq 0$ and $i_1 \neq i_2$, contradicting the assumption that $\cT$ is valid. 

    Since we have obtained a contradiction, $\cG$ must in fact be connected, which proves \ref{item:2-good-signs}. 
    
    Finally, \ref{item:2-good-signs} implies that every variable is contained in an even number of $(*_1)$, \ldots, $(*_t)$, $(*_\cT)$. This eliminates the possibility that two variables appear in exactly $3$ equations, finishing the proof of \ref{item:2-good-mi-vars}.
\end{proof}

\begin{lemma}\label{lem:min-impl-unique}
    Let $\cT$ be a $2$-good collection of linearly independent difference equalities. Then there is at most one difference equality that $\cT$ minimally implies. 
\end{lemma}

\begin{proof}
    If $\cT$ minimally implies a difference equality $(*_\cT)$, then by Lemma \ref{lem:2-good-min-impl}, every variable appears an even number of times among $\cT$ and $(*_\cT)$, so the variables appearing in $(*_\cT)$ are precisely the ones that appear in an odd number of equations in $\cT$. Then if $\cT$ minimally implied two distinct difference equalities, these difference equalities would have to involve the same $4$ variables, contradicting Lemma \ref{lem:2-eqns-5-vars}. 
\end{proof}

\subsection{Proof of Lemma \ref{lem:2s-certify-equality}} \label{subsec:2s-certify-equality}

In this subsection, we prove Lemma \ref{lem:2s-certify-equality} (which analyzes the $2$-good collections $\cS$ that achieve equality or near-equality in Lemma \ref{lem:2s-certify}). Recall that we proved Lemma \ref{lem:2s-certify} by drawing `boxes' around all subsets of $\cS$ that minimally imply some difference equality involving $x_i$, and showing that each connected `blob' formed by these boxes must involve reasonably few variables. The idea of our proof of Lemma \ref{lem:2s-certify-equality} is essentially that if $\cS$ is $2$-good, then Lemma \ref{lem:2-good-min-impl} heavily restricts the structure of these boxes and blobs --- it will immediately imply that all boxes have size $1$ or $3$, and we will later show that the boxes of size $3$ must be disjoint. This additional structure makes it much simpler to analyze when equality or near-equality holds in the bounds used to prove Lemma \ref{lem:2s-certify}. 

First, by Lemma \ref{lem:2-good-min-impl}\ref{item:2-good-mi-vars}, if $\{(*_1), \ldots, (*_t)\}$ is a $2$-good collection of linearly independent difference equalities involving $x_i$ that minimally implies a difference equality $(*)$ also involving $x_i$, then $t$ must be $1$ or $3$, since $x_i$ appears in all $t + 1$ equations. (In the terminology used in the proof of Lemma \ref{lem:2s-certify}, this states that all boxes have size $1$ or $3$.) The case where $t = 1$ is straightforward to analyze --- $(*)$ must be the same equation as $(*_1)$. We now analyze the case where $t = 3$. We call such a set $\{(*_1), (*_2), (*_3)\}$ a \emph{$3$-implication at $i$} (or simply a \emph{$3$-implication}, if we do not need to specify $i$), and we say it \emph{produces} the difference equality $(*)$. Note that by Lemma \ref{lem:min-impl-unique}, each $3$-implication produces only one difference equality. 

\begin{lemma}\label{lem:3-impls}
    Let $\cT$ be a $3$-implication and $(*)$ the difference equality it produces. Let $\cG_3(\cT)$ be the graph with vertex set $\cT \cup \{(*)\}$ where two equations are adjacent if they are difference-aligned. Then $\cG_3(\cT)$ must be a $4$-vertex path. 
\end{lemma}

\begin{proof}
    Let $\cT = \{(*_1), (*_2), (*_3)\}$, and suppose that $\cT$ is a $3$-implication at $1$ involving the $7$ variables $x_1$, \ldots, $x_7$ (by Lemma \ref{lem:2-good-min-impl}\ref{item:2-good-mi-vars}, any $3$-implication involves exactly $7$ variables). Let $*_1$, $*_2$, $*_3$, $*$ be the contents of $(*_1)$, $(*_2)$, $(*_3)$, $(*)$; we may assume that $x_1$ has a coefficient of $+1$ in each. Then by Lemma \ref{lem:2-good-min-impl}\ref{item:2-good-signs}, we can write ${*} = \eps_1{*_1} + \eps_2{*_2} + \eps_3{*_3}$ for some $\eps_1, \eps_2, \eps_3 \in \{\pm 1\}$. Exactly two of $\eps_1$, $\eps_2$, $\eps_3$ must be $+1$, so we can assume without loss of generality that  
    \begin{align}
        {*} = {*_1} + {*_2} - {*_3}.\label{eqn:3-impl}
    \end{align}

    First note that each of the $6$ variables $x_2$, \ldots, $x_7$ must be shared by exactly one pair of $(*_1)$, $(*_2)$, $(*_3)$, $(*)$; meanwhile, each of the $\binom{4}{2} = 6$ pairs of equations shares at most one such variable by Lemma \ref{lem:2-eqns-5-vars} (since the pair already shares $x_1$). This means each pair of equations must share \emph{exactly} one variable. 

    Then the variable shared by $(*_1)$ and $(*_2)$ must have opposite coefficients in $*_1$ and $*_2$, so we can assume without loss of generality that
    \begin{alignat*}{10}
        {*_1} &= x_1 && - x_2 - x_3 + x_4 && \\
        {*_2} &= x_1 && + x_2 && -x_5 - x_6.
    \end{alignat*}
    Then one of $x_5$ and $x_6$ must appear in $(*_3)$, and the other must appear in $(*)$. Both must appear with opposite sign as $x_1$ in their respective equations in order to cancel out of \eqref{eqn:3-impl}, so $(*_2)$ must be difference-aligned with both $(*_3)$ and $(*)$. For the same reason, $(*_1)$ is difference-aligned with whichever of $(*_3)$ and $(*)$ involves $x_3$ (and is not difference-aligned with either of the other two equations). Finally, $(*_3)$ and $(*)$ cannot be difference-aligned for the same reason that $(*_1)$ and $(*_2)$ cannot be --- the shared variable $x_7$ must have opposite coefficients in $*_3$ and $*_4$. So $\cG_3(\cT)$ is a $4$-vertex path. 
\end{proof}

Figure \ref{fig:3-impl-graph} shows an example of the graph $\cG_3(\cT)$. In fact, by slightly extending the argument in our proof of Lemma \ref{lem:3-impls}, one can show that \emph{any} $3$-implication must look like the one in Figure \ref{fig:3-impl-graph}, up to swapping which $3$-equation subset of the graph we consider to be $\cT$ (and, of course, renaming the variables). 

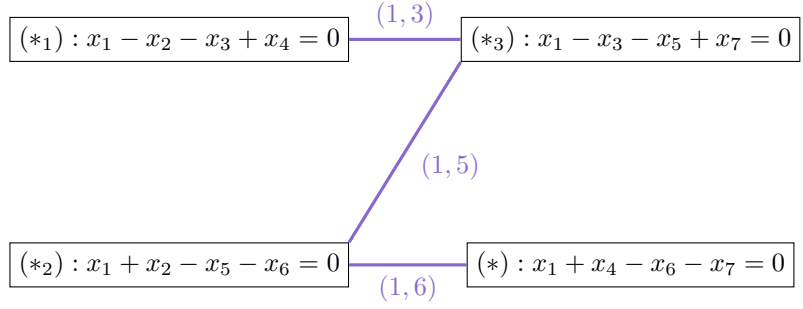
\begin{figure}[ht]
    \centering 
    \begin{tikzpicture}
        \node (1) [draw] at (0, 0) {$(*_1) : x_1 - x_2 - x_3 + x_4 = 0$};
        \node (2) [draw] at (0, -3) {$(*_2) : x_1 + x_2 - x_5 - x_6 = 0$};
        \node (3) [draw] at (6, 0) {$(*_3) : x_1 - x_3 - x_5 + x_7 = 0$};
        \node (4) [draw] at (6, -3) {$(*) : x_1 + x_4 - x_6 - x_7 = 0$};
        \begin{scope}[MediumPurple3, very thick]
            \draw (1) to node [midway, above] {$(1, 3)$} (3);
            \draw (2) to node [midway, below] {$(1, 6)$} (4);
            \draw (2.north east) to node [pos = 0.55, below right] {$(1, 5)$} (3.south west);
        \end{scope}
    \end{tikzpicture}
    \caption{This figure shows an example of a $3$-implication $\cT$ at $1$ and the corresponding graph $\cG_3(\cT)$; here each edge of $\cG_3(\cT)$ is labelled with the pair $(1, j)$ at which its endpoints are difference-aligned.}
    \label{fig:3-impl-graph}
\end{figure}

Lemma \ref{lem:3-impls} describes how an individual $3$-implication can look; the following lemma describes how two $3$-implications can interact with each other. 

\begin{lemma}\label{lem:3-impls-disjoint}
    Fix $i \in \{1, \ldots, k\}$. Let $\cS$ be a $2$-good collection of linearly independent difference equalities involving $x_i$, and suppose that $\cT_1$ and $\cT_2$ are distinct $3$-implications contained in $\cS$. Then $\cT_1$ and $\cT_2$ are disjoint (i.e., they have no equations in common). 
\end{lemma}

\begin{proof}
    Assume that $\cT_1$ and $\cT_2$ are not disjoint, so they have $1$ or $2$ equations in common. 

    First suppose they have exactly $1$ equation in common. Each of $\cT_1$ and $\cT_2$ involves $7$ variables, and the shared equation involves $4$; this means $\cT_1 \cup \cT_2$ involves at most $7 + 7 - 4 = 10$ variables. Then $\cT_1 \cup \cT_2$ contains $5$ linearly independent equations on at most $10$ variables; this contradicts the assumption that $\cS$ is $2$-light (by Remark \ref{rmk:c-light-alt}).

    Now suppose they have exactly $2$ equations in common. Then by Lemma \ref{lem:2-eqns-5-vars}, the $2$ shared equations together involve at least $6$ variables; this means $\cT_1 \cup \cT_2$ involves at most $7 + 7 - 6 = 8$ variables. But then $\cT_1 \cup \cT_2$ contains $4$ linearly independent equations on at most $8$ variables; this again contradicts the assumption that $\cS$ is $2$-light. 
\end{proof}

We now have all the pieces needed to prove Lemma \ref{lem:2s-certify-equality}. We first prove the following slightly more general statement, from which it will be easy to deduce Lemma \ref{lem:2s-certify-equality}. 

\begin{lemma}\label{lem:sprime-certify}
    Fix $i \in \{1, \ldots, k\}$, and let $\cS$ be a $2$-good collection consisting of $s$ linearly independent difference equalities that all involve $x_i$. Then we can find some $\cS'$, also consisting of $s$ linearly independent difference equalities that all involve $x_i$, such that $\cS'$ is equivalent to $\cS$ and such that for each index $j \neq i$, $\cS$ $i$-certifies $j$ if and only if $x_i$ and $x_j$ appear with opposite sign in an equation in $\cS'$.  
\end{lemma}

\begin{proof}
    Note that every difference equality involving $x_i$ implied by $\cS$ either is already present in $\cS$ or is produced by a $3$-implication. Furthermore, the $3$-implications in $\cS$ are all disjoint (by Lemma \ref{lem:3-impls-disjoint}), and each produces only one difference equality (by Lemma \ref{lem:min-impl-unique}). 

    We now define $\cS'$ in the following way: 
    \begin{itemize}
        \item For each equation $(*)$ in $\cS$ \emph{not} contained in any $3$-implication, we simply place $(*)$ into $\cS'$.
        \item For each $3$-implication $\cT \subseteq \cS$, let $(*_\cT)$ be the equation it produces. By Lemma \ref{lem:3-impls}, the graph $\cG_3(\cT)$ on $\cT \cup \{(*_\cT)\}$ is a $4$-vertex path by Lemma \ref{lem:3-impls}; we consider one vertex of degree $2$, and place the three \emph{other} equations in $\cT \cup \{(*_\cT)\}$ into $\cS'$. 
    \end{itemize}
    An example of this process is shown in Figure \ref{fig:sprime-ex}. 

    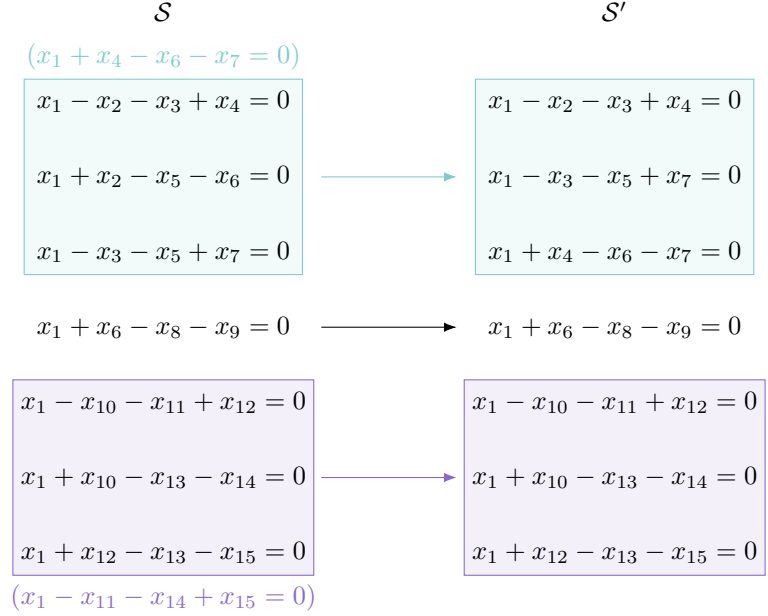
\begin{figure}[ht]
        \centering 
        \begin{tikzpicture}
            \filldraw [draw = DarkSlateGray3, fill = DarkSlateGray3!10] (-1.85, -2.3) rectangle (1.85, 0.3);
            \filldraw [draw = MediumPurple3, fill = MediumPurple3!10] (-2, -6.3) rectangle (2, -3.7);
            \node at (0, 0) {$x_1 - x_2 - x_3 + x_4 = 0$};
            \node at (0, -1) {$x_1 + x_2 - x_5 - x_6 = 0$};
            \node at (0, -2) {$x_1 - x_3 - x_5 + x_7 = 0$};
            
            \node at (0, -3) {$x_1 + x_6 - x_8 - x_9 = 0$};
            \node at (0, -4) {$x_1 - x_{10} - x_{11} + x_{12} = 0$};
            \node at (0, -5) {$x_1 + x_{10} - x_{13} - x_{14} = 0$};
            \node at (0, -6) {$x_1 + x_{12} - x_{13} - x_{15} = 0$};
            
            \node [DarkSlateGray3] at (0, 0.6) {$(x_1 + x_4 - x_6 - x_7 = 0)$};
            \node [MediumPurple3] at (0, -6.6) {$(x_1 - x_{11} - x_{14} + x_{15} = 0)$};
            \node at (0, 1.2) {$\cS$};
            \begin{scope}[xshift = 6cm]
                \filldraw [draw = DarkSlateGray3, fill = DarkSlateGray3!10] (-1.85, -2.3) rectangle (1.85, 0.3);
                \filldraw [draw = MediumPurple3, fill = MediumPurple3!10] (-2, -6.3) rectangle (2, -3.7);
                \node at (0, 1.2) {$\cS'$};
                \node at (0, 0) {$x_1 - x_2 - x_3 + x_4 = 0$};
                \node at (0, -1) {$x_1 - x_3 - x_5 + x_7 = 0$};
                \node at (0, -2) {$x_1 + x_4 - x_6 - x_7 = 0$};
                \node at (0, -3) {$x_1 + x_6 - x_8 - x_9 = 0$};
                \node at (0, -4) {$x_1 - x_{10} - x_{11} + x_{12} = 0$};
                \node at (0, -5) {$x_1 + x_{10} - x_{13} - x_{14} = 0$};
                \node at (0, -6) {$x_1 + x_{12} - x_{13} - x_{15} = 0$};
            \end{scope}
            \draw [DarkSlateGray3, -Latex] (2.1, -1) -- (3.9, -1);
            \draw [-Latex] (2.1, -3) -- (3.9, -3);
            \draw [MediumPurple3, -Latex] (2.1, -5) -- (3.9, -5);
        \end{tikzpicture}
        \caption{This figure shows an example of a possible collection $\cS$ (on the left-hand side) and one way to choose the corresponding $\cS'$ (on the right-hand side), with $i = 1$. The $3$-implications in $\cS$ are shown in boxes and labelled with the difference equalities they produce (shown in parentheses --- the equations in parentheses are not part of $\cS$). In this example, for the first $3$-implication (boxed in blue) we leave out $x_1 + x_2 - x_5 - x_6 = 0$, and for the second $3$-implication (boxed in purple) we leave out $x_1 - x_{11} - x_{14} + x_{15} = 0$.}
        \label{fig:sprime-ex}
    \end{figure}

    Then it is clear that $\cS'$ consists of $s$ linearly independent difference equalities involving $x_i$, and that $\cS'$ is equivalent to $\cS$ (note that if $\cT$ is a $3$-implication producing $(*_\cT)$, then all $3$-equation subsets of $\cT \cup \{(*_\cT)\}$ are equivalent). 
    
    Finally, we show that $\cS$ $i$-certifies an index $j$ if and only if $x_i$ and $x_j$ appear with opposite sign in some equation in $\cS'$. By definition, $\cS$ $i$-certifies $j$ if and only if $x_i$ and $x_j$ appear with opposite sign in some difference equality implied by $\cS$. Of course every equation in $\cS'$ is implied by $\cS$, so if $x_i$ and $x_j$ appear with opposite sign in an equation in $\cS'$, then $\cS$ $i$-certifies $j$. Conversely, if $(*)$ is a difference equality implied by $\cS$ that is \emph{not} in $\cS'$, then it must be a vertex of degree $2$ in $\cG_3(\cT)$ for some $3$-implication $\cT \subseteq \cS$. This means it is difference-aligned with two other equations in $\cG_3(\cT)$, so both variables $x_j$ that appear with opposite sign as $x_i$ in $(*)$ must also appear with opposite sign as $x_i$ in one of the other three equations in $\cG_3(\cT)$. (Note that the two equations difference-aligned with $(*)$ in $\cG_3(\cT)$ cannot both be difference-aligned with $(*)$ at the \emph{same} pair $(i, j)$, as then they would be difference-aligned with each other and $\cG_3(\cT)$ would contain a triangle.) All three of these equations are present in $\cS'$, so $\cS'$ has the desired property. 
\end{proof}

\begin{proof}[Proof of Lemma \ref{lem:2s-certify-equality}]
    Let $\cS'$ be the result of applying Lemma \ref{lem:sprime-certify} to $\cS$. Then it suffices to verify \ref{item:certify-2s} and \ref{item:certify-one-less}, given that $\cS$ $i$-certifies $j$ if and only if $x_i$ and $x_j$ appear with opposite sign in some equation in $\cS'$. 

    For \ref{item:certify-2s}, we are given that there are exactly $2s$ indices $j$ such that $x_i$ and $x_j$ appear with opposite sign in some equation in $\cS'$. Since each of the $s$ equations in $\cS'$ has exactly two variables that appear with opposite sign as $x_i$, this means for each such $j$, the variable $x_j$ appears with opposite sign in only one equation in $\cS'$, and therefore no two equations in $\cS'$ are difference-aligned. (Recall that by Lemma \ref{lem:2-eqns-5-vars}, since any two equations in $\cS'$ share $x_i$, if they are difference-aligned then they must be difference-aligned at $(i, j)$ for some $j$, meaning that some variable $x_j$ must appear with opposite sign as $x_i$ in both.)

    Similarly, for \ref{item:certify-one-less}, we are given that there are exactly $2s - 1$ indices $j$ such that $x_i$ and $x_j$ appear with opposite sign in some equation in $\cS'$. This means one such variable $x_j$ must appear with opposite sign as $x_i$ in exactly two equations in $\cS'$, and the rest must appear with opposite sign as $x_i$ in only one; so exactly one pair of equations in $\cS'$ is difference-aligned.
\end{proof}

\subsection{Proof of Lemma \ref{lem:impls-not-i}} \label{subsec:impls-not-i}

In this subsection, we prove Lemma \ref{lem:impls-not-i}. This proof will involve analyzing collections of linearly independent difference equalities involving $x_i$ that minimally imply some difference equality \emph{not} involving $x_i$. Similarly to in the previous subsection, the main idea is that Lemma \ref{lem:2-good-min-impl} heavily restricts how such a collection can look, which will allow us to bound the number of such collections. 

First, by Lemma \ref{lem:2-good-min-impl}\ref{item:2-good-mi-vars}, if $\{(*_1), \ldots, (*_t)\}$ is a $2$-good collection of linearly independent difference equalities involving $x_i$ that minimally implies some difference equality $(*)$ \emph{not} involving $x_i$, then $t$ must be $2$ or $4$ (as $x_i$ appears in exactly $t$ of these equations). We call such a collection $\{(*_1), \ldots, (*_t)\}$ a \emph{$2$-implication at $i$} (or simply a \emph{$2$-implication}, if we do not need to specify $i$) if $t = 2$, and a \emph{$4$-implication at $i$} (or simply a \emph{$4$-implication}) if $t = 4$; and we say it \emph{produces} the difference equality $(*)$. Again by Lemma \ref{lem:min-impl-unique}, any $2$-implication or $4$-implication produces only one difference equality. 

First, $2$-implications are easy to understand. 

\begin{lemma}\label{lem:2-impls}
    The $2$ equations of any $2$-implication are either difference-aligned or sum-aligned. 
\end{lemma}

\begin{proof}
    Let $\{(*_1), (*_2)\}$ be a $2$-implication at $i$, and let $(*)$ be the difference equality it produces. Let $*_1$, $*_2$, $*$ be the contents of $(*_1)$, $(*_2)$, $(*)$; we can assume without loss of generality that $x_i$ has a coefficient of $+1$ in $*_1$ and $*_2$, and that \[{*} = {*_1} - {*_2}.\] 
    
    Then $*_1$ and $*_2$ must share some variable other than $x_i$ (otherwise $(*)$ would involve $7$ variables), and this variable must have the same coefficient in both; this means $(*_1)$ and $(*_2)$ must be either difference-aligned (if this coefficient is $-1$) or sum-aligned (if this coefficient is $+1$). 
\end{proof}

We now turn to $4$-implications. To prove Lemma \ref{lem:impls-not-i}, we only need to consider $4$-implications in which at most one pair of equations is difference-aligned (since we assume that the collection $\cS$ in Lemma \ref{lem:impls-not-i} contains at most one pair of difference-aligned equations). 

\begin{lemma}\label{lem:4-impl}
    Let $\cT$ be a $4$-implication such that at most one pair of equations in $\cT$ is difference-aligned. Then exactly one pair of equations in $\cT$ is difference-aligned, and exactly one pair is sum-aligned. 
\end{lemma}

\begin{proof}
    Let $\cT = \{(*_1), (*_2), (*_3), (*_4)\}$, and let $(*)$ be the difference equality produced by $\cT$. Suppose that $\cT$ is a $4$-implication at $1$ involving the $9$ variables $x_1$, \ldots, $x_9$ (by Lemma \ref{lem:2-good-min-impl}\ref{item:2-good-mi-vars}, any $4$-implication involves exactly $9$ variables). Let $*_1$, $*_2$, $*_3$, $*_4$, $*$ be the contents of $(*_1)$, $(*_2)$, $(*_3)$, $(*_4)$, $(*)$; we can assume without loss of generality that the coefficient of $x_1$ in each of $*_1$, $*_2$, $*_3$, $*_4$ is $+1$. Then by Lemma \ref{lem:2-good-min-impl}\ref{item:2-good-signs} we have ${*} = \eps_1{*_1} + \cdots + \eps_4{*_4}$ for some $\eps_1, \ldots, \eps_4 \in \{\pm 1\}$; exactly two of these coefficients must be $+1$, so we can assume without loss of generality that \[{*} = {*_1} + {*_2} - {*_3} - {*_4}.\] 

    First we claim that $*_1$ and $*_2$ must share exactly one variable other than $x_1$. To see this, by Lemma \ref{lem:2-eqns-5-vars} they cannot share more than one such variable; now assume for contradiction that they share none. Then we can assume without loss of generality that 
    \begin{alignat*}{10}
        {*_1} &= x_1 - x_2 - x_3 + x_4 && \\
        {*_2} &= x_1 && - x_5 - x_6 + x_7.
    \end{alignat*} 
    Then at least two of $x_2$, $x_3$, $x_5$, and $x_6$ must appear in $*_3$ or $*_4$ with coefficient $-1$, as any of these variables that does not do so will appear in $*$ with coefficient $-1$ (and $*$ can only have two variables with coefficient $-1$). But this means at least two pairs of $(*_1)$, $(*_2)$, $(*_3)$, $(*_4)$ are difference-aligned, which is a contradiction. 

    So $*_1$ and $*_2$ share exactly one variable other than $x_1$, and this variable must have coefficient $-1$ in one and $+1$ in the other; so we can assume without loss of generality that 
    \begin{alignat*}{10}
        {*_1} &= x_1 - x_2 - x_3 + x_4 && \\
        {*_2} &= x_1 + x_2 && - x_5 - x_6.
    \end{alignat*}

    As before, at most two of $x_3$, $x_5$, and $x_6$ can appear in $*$ (as each would have coefficient $-1$), while any of these variables that does \emph{not} appear in $*$ must appear in $*_3$ or $*_4$ with coefficient $-1$, producing a difference-aligned pair of equations. Since we have at most one difference-aligned pair, this means \emph{exactly} two of $x_3$, $x_5$, and $x_6$ must appear in $*$, and there is exactly one difference-aligned pair of equations (the pair sharing the one variable among $x_3$, $x_5$, and $x_6$ that does not appear in $*$). 

    Finally, $*_3$ and $*_4$ must also share exactly one variable other than $x_1$ (for the same reason that $*_1$ and $*_2$ do); let this variable be $x_7$, and assume that $x_7$ has coefficient $-1$ in $*_3$ and $+1$ in $*_4$. Then $x_8$ and $x_9$ must both appear in $*$ (if one did not, then it would also be shared by $*_3$ and $*_4$). Since only $4$ variables appear in $*$, this means $x_4$ does not; so it must appear with coefficient $+1$ in $*_3$ (it cannot appear with coefficient $+1$ in $*_4$, as $x_7$ does). So there is exactly one sum-aligned pair, namely $(*_1)$ and $(*_3)$. 
\end{proof}

\begin{example} \label{ex:4-impl}
    We illustrate Lemma \ref{lem:4-impl} for two examples of $4$-implications at $1$; as in the proof of Lemma \ref{lem:4-impl}, we let our $4$-implication be $\{(*_1), (*_2), (*_3), (*_4)\}$ and the difference equality it produces be $(*)$, and we let the contents of $(*_1)$, $(*_2)$, $(*_3)$, $(*_4)$, $(*)$ be $*_1$, $*_2$, $*_3$, $*_4$, $*$. 
    
    One example of a $4$-implication with only one difference-aligned pair of equations is given by 
    \begin{alignat*}{10}
        *_1 &= x_1 && - x_2 && - x_3 &&+ x_4 && && && && && \\
        *_2 &= x_1 && + x_2 && && && - x_5 && - x_6 && && && \\
        *_3 &= x_1 && && && + x_4 && - x_5 && && - x_7 && && \\
        *_4 &= x_1 && && && && && && + x_7 && - x_8 && - x_9 \\ 
        \cline{1-19} 
        * &= && && - x_3 && && && - x_6 && && + x_8 && + x_9
    \end{alignat*} 
    (where we have ${*} = {*_1} + {*_2} - {*_3} - {*_4}$). Here the one difference-aligned pair is $\{(*_2), (*_3)\}$, and the one sum-aligned pair is $\{(*_1), (*_3)\}$. 

    Another example of a $4$-implication with only one difference-aligned pair of equations is 
    \begin{alignat*}{10}
        *_1 &= x_1 && - x_2 && - x_3 &&+ x_4 && && && && && \\
        *_2 &= x_1 && + x_2 && && && - x_5 && - x_6 && && && \\
        *_3 &= x_1 && && && + x_4 && && && - x_7 && - x_8 && \\
        *_4 &= x_1 && && && && - x_5 && && + x_7 && && - x_9 \\ 
        \cline{1-19} 
        * &= && && - x_3 && && && - x_6 && && + x_8 && + x_9
    \end{alignat*} 
    (where we again have ${*} = {*_1} + {*_2} - {*_3} - {*_4}$). Here the one difference-aligned pair is $\{(*_2), (*_4)\}$, and the one sum-aligned pair is $\{(*_1), (*_3)\}$. 

    In fact, by slightly extending the arguments in Lemmas \ref{lem:4-impl} and \ref{lem:4-impls-7}, it can be shown that any $4$-implication with only one difference-aligned pair of equations must be of one of these two forms. 
\end{example}

Lemma \ref{lem:4-impl} describes how an individual $4$-implication can look; the next three lemmas analyze how $4$-implications can interact with each other. 

\begin{lemma}\label{lem:4-impls-overlap}
    Fix $i \in \{1, \ldots, k\}$, and let $\cS$ be a $2$-good collection of linearly independent difference equalities involving $x_i$ such that only one pair of equations in $\cS$ is difference-aligned. Suppose that $\cT_1$ and $\cT_2$ are distinct $4$-implications contained in $\cS$. Then $\cT_1$ and $\cT_2$ share exactly $3$ equations, which involve exactly $7$ variables. 
\end{lemma}

\begin{proof}
    First, by Lemma \ref{lem:4-impl}, $\cT_1$ and $\cT_2$ must both contain the one difference-aligned pair of equations in $\cS$; this means they have either $2$ or $3$ equations in common. 

    First assume for contradiction that they have $2$ equations in common. Then by Lemma \ref{lem:2-eqns-5-vars}, these $2$ equations together involve at least $6$ variables; since $\cT_1$ and $\cT_2$ each involves $9$ variables, this means $\cT_1 \cup \cT_2$ involves at most $9 + 9 - 6 = 12$ variables. But then $\cT_1 \cup \cT_2$ consists of $6$ linearly independent equations on at most $12$ variables, contradicting the assumption that $\cS$ is $2$-light (by Remark \ref{rmk:c-light-alt}). 

    So $\cT_1$ and $\cT_2$ have $3$ equations in common, and these $3$ equations must involve at least $7$ variables (again by the assumption that $\cS$ is $2$-light and Remark \ref{rmk:c-light-alt}). If they involved at least $8$ variables, then $\cT_1 \cup \cT_2$ would consist of $5$ linearly independent equations on at most $9 + 9 - 8 = 10$ variables, again contradicting the assumption that $\cS$ is $2$-light; so they must involve \emph{exactly} $7$. 
\end{proof}

\begin{lemma}\label{lem:4-impls-7}
    Fix $i \in \{1, \ldots, k\}$, and let $\cT$ be a $4$-implication at $i$ such that only one pair of equations in $\cT$ is difference-aligned. Suppose that there is some $3$-equation subset $\cT' \subseteq \cT$ that involves exactly $7$ variables. Then there is only one such subset, and this subset $\cT'$ forms a $3$-implication at $i$ producing some difference equality (involving $x_i$) that is sum-aligned with the fourth equation. 
\end{lemma}

For example, in the first $4$-implication in Example \ref{ex:4-impl}, $\{(*_1), (*_2), (*_3)\}$ is such a subset; the second $4$-implication in Example \ref{ex:4-impl} has no such subset. 

\begin{proof}
    Let $\cT = \{(*_1), (*_2), (*_3), (*_4)\}$, let $(*_\cT)$ be the difference equality it produces, and let $*_1$, $*_2$, $*_3$, $*_4$, $*_\cT$ be the contents of $(*_1)$, $(*_2)$, $(*_3)$, $(*_4)$, $(*_\cT)$. By Lemma \ref{lem:2-good-min-impl}\ref{item:2-good-signs}, we can write \[{*_\cT} = \eps_1{*_1} + \cdots + \eps_4{*_4}\] for some $\eps_1, \ldots, \eps_4 \in \{\pm 1\}$. 

    Assume that $\cT' = \{(*_1), (*_2), (*_3)\}$ involves $7$ variables. Let \[{*_{\cT'}} = \eps_1{*_1} + \eps_2{*_2} + \eps_3{*_3} = {*_\cT} - \eps_4{*_4},\] and let $(*_{\cT'})$ be the equation ${*_{\cT'}} = 0$. 
    
    We claim that $(*_{\cT'})$ is a difference equality involving $x_i$. To see this, first note that $*_1$, $*_2$, and $*_3$ together contain $3 \cdot 3 = 9$ `slots' for variables other than $x_i$, so since they contain $6$ variables other than $x_i$ and each appears once or twice, $3$ of these variables must appear once and $3$ must appear twice. The $2$ variables that do not appear in $*_1$, $*_2$, $*_3$ clearly cannot appear in $*_{\cT'}$; meanwhile, the $3$ variables that appear twice in $*_1$, $*_2$, $*_3$ cannot appear in $*_\cT$ or $*_4$, and therefore cannot appear in $*_{\cT'}$ either. On the other hand, the $3$ variables that appear once in $*_1$, $*_2$, $*_3$, as well as $x_i$, must each appear in exactly one of $*_\cT$ and $*_4$; so these four variables all appear in $*_{\cT'}$ with coefficient $\pm 1$. Finally, the coefficients of $*_{\cT'}$ must sum to $0$; this means $(*_{\cT'})$ is a difference equality. 

    So $\{(*_1), (*_2), (*_3)\}$ is a $3$-implication at $i$ producing $(*_{\cT'})$. By Lemma \ref{lem:3-impls-disjoint}, this means no other subset can contain $7$ variables (as then it would also form a $3$-implication, so we would have two $3$-implications that are not disjoint). 

    Finally, we need to show that $(*_{\cT'})$ is sum-aligned with $(*_4)$. First note that ${*_\cT} = {*_{\cT'}} + \eps_4{*_4}$, so $\{(*_{\cT'}), (*_4)\}$ is a $2$-implication; by Lemma \ref{lem:2-impls}, this means $(*_{\cT'})$ and $(*_4)$ are either difference-aligned or sum-aligned. But by Lemma \ref{lem:3-impls} applied to $\cT'$, the graph $\cG_3(\cT')$ is a $4$-vertex path on $(*_1)$, $(*_2)$, $(*_3)$, $(*_\cT')$. Since at most one pair of $(*_1)$, $(*_2)$, $(*_3)$ is difference-aligned by assumption (and therefore there is at most one edge between them), this means \emph{exactly} one pair is, and $(*_{\cT'})$ must be one of the vertices of degree $2$ in this path --- so both variables that appear with opposite sign as $x_i$ in $(*_{\cT'})$ must also appear with opposite sign as $x_i$ in one of $(*_1)$, $(*_2)$, $(*_3)$. Then if $(*_4)$ were difference-aligned with $(*_{\cT'})$, it would also be difference-aligned with one of $(*_1)$, $(*_2)$, $(*_3)$, producing a second difference-aligned pair among $(*_1)$, $(*_2)$, $(*_3)$, $(*_4)$. This is impossible, so $(*_{\cT'})$ and $(*_4)$ must instead be sum-aligned.  
\end{proof}

\begin{lemma}\label{lem:many-4-impls}
    Fix $i \in \{1, \ldots, k\}$, and let $\cS$ be a $2$-good collection of linearly independent difference equalities involving $x_i$ such that at most one pair is difference-aligned. Suppose that $\cT_1$, \ldots, $\cT_r$ are distinct $4$-implications contained in $\cS$. Then there exist three equations present in all of them, and their fourth equations are all sum-aligned. 
\end{lemma}

\begin{proof}
    If $r = 1$, the statement is vacuously true; now assume $r \geq 2$. Then by Lemma \ref{lem:4-impls-overlap}, each of $\cT_2$, \ldots, $\cT_r$ shares exactly $3$ equations with $\cT_1$, and these $3$ equations involve exactly $7$ variables. Furthermore, by Lemma \ref{lem:4-impls-7}, $\cT_1$ only has one $3$-equation subset $\cT'$ involving exactly $7$ variables, so it must share the same subset $\cT'$ with each of $\cT_2$, \ldots, $\cT_r$. 

    Then each of $\cT_1$, \ldots, $\cT_r$ contains the $3$-equation subset $\cT'$. Furthermore, by Lemma \ref{lem:4-impls-7}, $\cT'$ must form a $3$-implication at $i$ producing a difference equality $(*_\cT')$ involving $x_i$, and the fourth equations of $\cT_1$, \ldots, $\cT_r$ must all be sum-aligned with $(*_{\cT'})$. This means these fourth equations are also sum-aligned with each other (recall that since these equations and $(*_{\cT'})$ all share $x_i$, by Lemma \ref{lem:2-eqns-5-vars} if two are sum-aligned then they must be sum-aligned at $(i, j)$ for some $j$). 
\end{proof}

We are now ready to prove Lemma \ref{lem:impls-not-i}. 

\begin{proof}[Proof of Lemma \ref{lem:impls-not-i}]
    Write down all the equations in $\cS$, and split them into \emph{stacks} based on which variable occurs with the same sign as $x_i$ (in other words, two equations are sum-aligned if and only if they are in the same stack). Suppose that we have $r$ stacks with sizes $s_1$, \ldots, $s_r$, so that $s_1 + \cdots + s_r = s$. 

    Every difference equality implied by $\cS$ that does not involve $x_i$ must be produced by a $2$-implication or a $4$-implication at $i$. Furthermore, each $2$-implication and each $4$-implication can only produce one difference equality. So in order to bound the number of difference equalities implied by $\cS$ that do not involve $x_i$, it suffices to bound the total number of $2$-implications and $4$-implications at $i$ contained in $\cS$. 

    First, we prove \ref{item:all-sum}: the given condition means that there is only one stack, and therefore there are exactly $\binom{s}{2}$ $2$-implications formed by sum-aligned pairs of equations. Meanwhile, since two equations cannot be both sum-aligned and difference-aligned, there are no difference-aligned pairs of equations in $\cS$; by Lemma \ref{lem:4-impl}, this also means there are no $4$-implications in $\cS$. So the total number of $2$-implications and $4$-implications is precisely $\binom{s}{2}$, as desired.  

    Next, we prove \ref{item:no-diff}: now we are given that there are at least two stacks, and that no two equations in $\cS$ are difference-aligned. Then for the same reason as in the previous case, the only $2$-implications and $4$-implications contained in $\cS$ are the $2$-implications formed by sum-aligned pairs, of which there are \[\binom{s_1}{2} + \cdots + \binom{s_r}{2}.\] Since the function $x \mapsto \binom{x}{2}$ is convex and there are at least two stacks (of sizes at least $1$), this is at most \[\binom{s - 1}{2} + \binom{1}{2} = \binom{s - 1}{2}.\] So the total number of $2$-implications and $4$-implications is at most $\binom{s - 1}{2}$, as desired.  

    Finally, we prove \ref{item:one-diff}: we are given that exactly one pair of equations in $\cS$ is difference-aligned; since two equations cannot be both difference-aligned and sum-aligned, this means there are at least two stacks. 
    
    First consider the case where $\cS$ does not contain any $4$-implications. Then $\cS$ contains $\binom{s_1}{2} + \cdots + \binom{s_r}{2}$ $2$-implications in which the two equations are sum-aligned, and one where the two equations are difference-aligned. So the number of $2$-implications contained in $\cS$ is \[\binom{s_1}{2} + \cdots + \binom{s_r}{2} + 1 \leq \binom{s - 1}{2} + 1.\]

    Now consider the case where $\cS$ \emph{does} contain a $4$-implication. By Lemma \ref{lem:4-impl}, any $4$-implication must span exactly $3$ stacks (as it contains exactly one sum-aligned pair), so there are at least $3$ nonempty stacks. Furthermore, by Lemma \ref{lem:many-4-impls}, there exist three equations present in all $4$-implications, and there is some fixed stack --- without loss of generality, the stack of size $s_1$ --- that the fouth equation of each $4$-implication belongs to. Then the number of $4$-implications is at most $s_1$. So the total number of $2$-implications and $4$-implications is at most \[\binom{s_1}{2} + \cdots + \binom{s_r}{2} + 1 + s_1 = \binom{s_1 + 1}{2} + \binom{s_2}{2} + \cdots + \binom{s_r}{2} + 1.\] Since $(s_1 + 1) + s_2 + \cdots + s_r = s + 1$ and there are at least three stacks (of sizes at least $1$), this is at most \[\binom{s - 1}{2} + \binom{1}{2} + \binom{1}{2} + 1 = \binom{s - 1}{2} + 1.\] 
    
    So in either case, the total number of $2$-implications and $4$-implications is at most $\binom{s - 1}{2} + 1$. 
\end{proof} 

Thus we have shown the technical lemmas we used to prove Lemma \ref{lem:odd-bounding-diffs}, which completes the proof of Proposition \ref{prop:odd-upper-bound}.

\section{Proof of Theorem \ref{thm:lower-bound}: A family of lower bounds on \texorpdfstring{$g(n, k, \ell)$}{g(n, k, l)}} \label{sec:lower-bound}

In this section, we prove Theorem \ref{thm:lower-bound}, our lower bounds on $g(n, k, \ell)$ for certain values of $\ell$. The idea of the proof is as follows: we show that if $\abs{A - A}$ is too small, then we can find a certain $k$-element structure $A'$ in $A$ that satisfies $\abs{A' - A'} \leq \ell - 1$ for the value of $\ell$ in Theorem \ref{thm:lower-bound}, which contradicts the $(k, \ell)$-local property. Finding this $k$-element structure will involve finding a sum or difference that is repeated many times. For this, we will need the following lemma. Clearly, a set with small difference set must have a difference that is repeated many times; the following lemma states that such a set must also have a \emph{sum} that is repeated many times. 

\begin{lemma}\label{lem:repeated-sums}
    If $A \subseteq \RR$ has $n$ elements, then \[\max_{c \in \RR} \left(\#\{(a, b) \in A^2 \mid a + b = c\}\right) \cdot \#\{a - b \mid (a, b) \in A^2\} \geq n^2.\] 
\end{lemma}

Note that $\{a - b \mid (a, b) \in A^2\}$ includes \emph{all} differences between two elements of $A$, not just the positive differences; in particular, we have $\#\{a - b \mid (a, b) \in A^2\} = 2\abs{A - A} + 1$.

\begin{proof}
    Let $m = \max_{c \in \RR} \left( \#\{(a, b) \in A^2 \mid a + b = c\}\right)$ and $d = \#\{a - b \mid (a, b) \in A^2\}$. Let \[S = \{(a_1, a_2, a_3, a_4) \in A^4 \mid a_1 - a_2 = a_3 - a_4\}.\] We will count $\abs{S}$ (the additive energy of $A$) in two ways, one in terms of $m$ and the other in terms of $d$. 

    First note that any $(a_1, a_2, a_3, a_4) \in S$ satisfies $a_1 + a_4 = a_2 + a_3$. So if we try to choose an element of $S$ by first choosing $a_1$ and $a_4$ (which can be done in $n^2$ ways), then there are at most $m$ ways to choose $a_2$ and $a_3$. This means $\abs{S} \leq n^2m$. 

    Now, suppose that the $d$ differences $a - b$ over all $n^2$ pairs $(a, b) \in A^2$ occur with multiplicities $x_1$, \ldots, $x_d$ respectively, so that $x_1 + \cdots + x_d = n^2$. Then by first fixing some difference and then choosing $(a_1, a_2)$ and $(a_3, a_4)$ to be two pairs achieving that difference, we have $\abs{S} = x_1^2 + \cdots + x_d^2$. By the Cauchy--Schwarz inequality, this means $\abs{S} \geq \frac{n^4}{d}$. 

    Combining these two bounds gives $\frac{n^4}{d} \leq \abs{S} \leq n^2m$, so $md \geq n^2$. 
\end{proof}

The next lemma states that we can find a certain structure in any set $A$ with sufficiently small $\abs{A - A}$. Intuitively, this structure consists of $2s$ congruent affine $(t - 1)$-cubes, whose centers form $s$ pairs that all have the same sum. 

\begin{lemma}\label{lem:find-cube}
    Let $s$ and $t$ be positive integers. Suppose that $n > 1$ and that $A$ is an $n$-element set such that \[\abs{A - A} \leq \left(\frac{n}{8s}\right)^{1 + \frac{1}{2^t - 1}}.\] Then there exist $d_1, \ldots, d_{t - 1} \in \RR_{> 0}$ and $a_1, \ldots, a_{2s} \in A$ such that $a_1 + a_{s + 1} = a_2 + a_{s + 2} = \cdots = a_s + a_{2s}$, and the $2^t\cdot s$ numbers \[a_i + \eps_1d_1 + \cdots + \eps_{t - 1}d_{t - 1}\] over all $1 \leq i \leq 2s$ and $\eps_1, \ldots, \eps_{t - 1} \in \{0, 1\}$ are all distinct and belong to $A$. 
\end{lemma}

\begin{proof}
    We use induction on $t$. First note that since $n > 1$, we must have $\abs{A - A} \geq 1$.

    For the base case $t = 1$ (in which case we are given $\abs{A - A} \leq \frac{n^2}{64s^2}$ and we only need to find distinct $a_1, \ldots, a_{2s} \in A$ such that $a_1 + a_{s + 1} = \cdots = a_s + a_{2s}$), we have \[\#\{a - b \mid (a, b) \in A^2\} = 2\abs{A - A} + 1 \leq 3\abs{A - A} \leq \frac{3n^2}{64s^2} \leq \frac{n^2}{3s},\] so by Lemma \ref{lem:repeated-sums} there must exist some $c \in \RR$ for which $\#\{(a, b) \in A^2 \mid a + b = c\} \geq 3s$. Fix some $c \in \RR$ with this property. Then at most one of these pairs $(a, b)$ can have $a = b$, and of the remaining pairs, any two are either disjoint or the same up to order (i.e., $(a_2, b_2) = (b_1, a_1)$). So there must exist at least $\frac{3s - 1}{2} \geq s$ disjoint pairs of distinct elements $(a, b) \in A^2$ with $a + b = c$, and letting $(a_1, a_{s + 1})$, \ldots, $(a_s, a_{2s})$ be $s$ of these pairs gives the desired result. 

    For the inductive step, first note that the given condition on $\abs{A - A}$ can be rewritten as 
    \begin{align}
        \abs{A} \geq 8s \cdot \abs{A - A}^{1 - \frac{1}{2^t}}.\label{eqn:A-8s-hyp}
    \end{align} 
    This form of the condition will be more useful when we apply the inductive hypothesis.
    
    Assume that $t \geq 2$ and that we have proven the statement for $s$ and $t - 1$ (for all $n > 1$ and all $n$-element sets $A$ with the stated property); we will then  prove the statement for $s$ and $t$. First, since there are $\binom{n}{2} \geq \frac{n^2}{4}$ total positive differences between pairs of elements in $A$ (counting multiplicity), some difference must be repeated at least \[\frac{\binom{n}{2}}{\abs{A - A}} \geq \frac{n^2}{4\cdot \abs{A - A}} \geq \frac{64s^2 \cdot \abs{A - A}^{2 - \frac{1}{2^{t - 1}}}}{4\cdot \abs{A - A}} = 16s^2\cdot \abs{A - A}^{1 - \frac{1}{2^{t - 1}}} \geq 16s\cdot \abs{A - A}^{1 - \frac{1}{2^{t - 1}}}\] times; let $d_{t - 1} > 0$ be such a difference. For notational convenience, define \[m = 8s\cdot \abs{A - A}^{1 - \frac{1}{2^{t - 1}}},\] so that there are at least $2m$ pairs $(a, b) \in A^2$ with $a - b = d_{t - 1}$ (and $a \neq b$ in all such pairs). We claim that there are at least $m$ such \emph{disjoint} pairs. To see this, consider the graph on all such pairs $(a, b)$ in which two pairs are adjacent if they have an element in common. This graph must be a collection of vertex-disjoint paths (as all vertices have degree at most $2$, and there are no cycles), so we can obtain an independent set consisting of at least half the vertices (by taking alternating vertices on each path, starting with one endpoint), which provides at least $m$ disjoint pairs $(a, b) \in A^2$ with $a - b = d_{t - 1}$. 

    Now let $A'$ be the set of values of $a$ across these pairs, so that $A'$ and $A' + d_{t - 1}$ are disjoint and both contained in $A$, and \[\abs{A'} \geq m = 8s\cdot \abs{A - A}^{1 - \frac{1}{2^{t - 1}}} \geq 8s \cdot \abs{A' - A'}^{1 - \frac{1}{2^{t - 1}}}.\]
    
    Then we can apply the inductive hypothesis to $A'$, as $A'$ satisfies the condition \eqref{eqn:A-8s-hyp} for $t - 1$ (and we have $\abs{A'} \geq m \geq 8s > 1$). This gives that we can find $a_1, \ldots, a_{2s} \in A'$ and $d_1, \ldots, d_{t - 2} \in \RR_{>0}$ such that $a_1 + a_{s + 1} = \cdots = a_s + a_{2s}$ and the $2^{t - 1}\cdot s$ numbers \[a_i + \eps_1d_1 + \cdots + \eps_{t - 2}d_{t - 2}\] are all distinct and belong to $A'$. Then the $2^{t - 1}\cdot s$ numbers \[a_i + \eps_1d_1 + \cdots + \eps_{t - 2}d_{t - 2} + d_{t - 1}\] are all distinct and belong to $A' + d_{t - 1}$. Since $A'$ and $A' + d_{t - 1}$ are disjoint and both contained in $A$, all $2^t \cdot s$ numbers are distinct and belong to $A$, as desired. 
\end{proof}

The following lemma provides an upper bound on the number of distinct differences present in a structure of the form given by Lemma \ref{lem:find-cube}.

\begin{lemma}\label{lem:cube-few-diffs}
    Let $a_1$, \ldots, $a_{2s}$, $d_1$, \ldots, $d_{t - 1}$ be fixed real numbers such that $a_1 + a_{s + 1} = \cdots = a_s + a_{2s}$, and let $A' = \{a_i + \eps_1d_1 + \cdots + \eps_{t - 1}d_{t - 1} \mid 1 \leq i \leq 2s \text{ and }  \eps_1, \ldots, \eps_{t - 1} \in \{0, 1\}\}$. Then \[\abs{A' - A'} \leq 3^{t - 1}s^2 + \frac{3^{t - 1} - 1}{2}.\] 
\end{lemma}

\begin{proof}
    Any difference $a - b$ for $(a, b) \in A^2$, including negative differences and zero, is of the form \[(a_{i_1} - a_{i_2}) + \lambda_1d_1 + \cdots + \lambda_{t - 1}d_{t - 1},\] where $1 \leq i_1, i_2 \leq 2s$ and $\lambda_1, \ldots, \lambda_{t - 1} \in \{-1, 0, 1\}$. The argument in Remark \ref{rmk:bounding-diffs-equality} (with $k = 2s$ and $d = s + 1$) shows that there are at most $s^2$ positive differences among $a_1$, \ldots, $a_{2s}$, and therefore at most $2s^2 + 1$ possible values of $a_{i_1} - a_{i_2}$ (including negative values and zero). Meanwhile, there are $3^{t - 1}$ choices of values for $\lambda_1$, \ldots, $\lambda_{t - 1}$; this means there are at most $3^{t - 1}(2s^2 + 1)$ such differences. Since this includes negative differences and zero, this means \[\abs{A' - A'} \leq \frac{3^{t - 1}(2s^2 + 1) - 1}{2} = 3^{t - 1}s^2 + \frac{3^{t - 1} - 1}{2}.\qedhere\] 
\end{proof}

Together, Lemmas \ref{lem:find-cube} and \ref{lem:cube-few-diffs} immediately imply Theorem \ref{thm:lower-bound}. 

\begin{proof}[Proof of Theorem \ref{thm:lower-bound}] 
    Let $s = \frac{k}{2^t}$ and $\ell = 3^{t - 1}s^2 + \frac{3^{t - 1} + 1}{2}$; then we claim that for all $n$, we have \[g(n, k, \ell) > \left(\frac{n}{8s}\right)^{1 + \frac{1}{2^t - 1}}.\] Assume not; then there exists an $n$-element set $A$ with $\abs{A - A} \leq \left(\frac{n}{8s}\right)^{1 + \frac{1}{2^t - 1}}$ satisfying the $(k, \ell)$-local property. But by Lemma \ref{lem:cube-few-diffs} we can find a $k$-element subset $A' \subseteq A$ of the form \[A' = \{a_i + \eps_1d_1 + \cdots + \eps_{t - 1}d_{t - 1} \mid 1 \leq i \leq 2s \text{ and } \eps_1, \ldots, \eps_{t - 1} \in \{0, 1\}\}\] for some $a_1$, \ldots, $a_{2s}$, $d_1$, \ldots, $d_{t - 1}$ satisfying $a_1 + a_{s + 1} = \cdots = a_{s} + a_{2s}$, and by Lemma \ref{lem:cube-few-diffs} we must have $\abs{A' - A'} \leq \ell - 1$, contradicting the assumption that $A$ satisfies the $(k, \ell)$-local property. 
\end{proof}

\section{Discussion of consequences} \label{sec:consequences}

In this section, we deduce and discuss the more intuitive consequences on the behavior of $g(n, k, \ell)$ stated in the beginning of Subsection \ref{subsec:results} (i.e., Corollaries \ref{cor:even-qt}, \ref{cor:odd-qt}, \ref{cor:nc-threshold}, \ref{cor:poly-bounds}, and \ref{cor:sk-loglogk}). 

\subsection{The quadratic threshold}

We begin by discussing our results on the quadratic threshold (i.e., Corollaries \ref{cor:even-qt} and \ref{cor:odd-qt}). 

When $k$ is even, Theorems \ref{thm:upper-bound} and \ref{thm:lower-bound} together immediately imply Corollary \ref{cor:even-qt}:
\begin{itemize}
    \item Applying Theorem \ref{thm:upper-bound} with $c = 2$ gives that $g(n, k, \frac{k^2}{4}) = o(n^2)$. 
    \item Applying Theorem \ref{thm:lower-bound} with $t = 1$ gives that $g(n, k, \frac{k^2}{4} + 1) = \Omega(n^2)$. 
\end{itemize}
In particular, these two bounds together mean that the quadratic threshold is exactly $\frac{k^2}{4} + 1$. 

\begin{remark}\label{rmk:why-qt}
    We now comment on why the proofs of Theorems \ref{thm:upper-bound} and \ref{thm:lower-bound} give matching values of $\ell$ when $k$ is even (thus allowing us to determine the exact quadratic threshold). To make this discussion more convenient, for a $k$-configuration $\cC$ we use $d(\cC)$ to denote the number of distinct differences in a generic solution to $\cC$ (it can be shown that $d(\cC)$ is well-defined, and that \emph{any} solution to $\cC$ has at most $d(\cC)$ distinct differences). 
    
    Let $k = 2s$. Then the idea of our proof of Theorem \ref{thm:lower-bound} for $t = 1$ is essentially that if $\abs{A - A}$ is small compared to $n^2$, then the specific $k$-configuration \[\cC_{\text{sum}} = \{x_1 + x_{s + 1} = x_2 + x_{s + 2} = \cdots = x_s + x_{2s}\}\] must occur in $A$. This gives a lower bound of $g(n, k, \ell) = \Omega(n^2)$ where $\ell = d(\cC_{\text{sum}}) + 1$ (because if $\abs{A - A}$ is small, then the occurrence of $\cC_{\text{sum}}$ in $A$ has at most $d(\cC_{\text{sum}}) < \ell$ distinct differences, contradicting the $(k, \ell)$-local property). 

    On the other hand, when we prove Theorem \ref{thm:upper-bound} for $c = 2$, we show that there exist $n$-element sets $A$ with $o(n^2)$ distinct differences such that every $k$-configuration that occurs in $A$ is $2$-good; this gives a bound of $g(n, k, \ell) = o(n^2)$ where $\ell$ is the minimum value of $d(\cC)$ over all $2$-good $k$-configurations $\cC$. 
    
    But as stated in Remark \ref{rmk:bounding-diffs-equality} (with $k = 2s$ and $d = s + 1$), this minimum is achieved by $\cC_{\text{sum}}$ (since it is an equality case of Lemma \ref{lem:bounding-diffs}, which gives a lower bound on $d(\cC)$ for $2$-good $\cC$). So the value of $\ell$ in our bound of $g(n, k, \ell) = o(n^2)$ actually ends up being $d(\cC_{\text{sum}})$; this is why the values of $\ell$ in the two bounds differ by exactly $1$. 
    
    In other words, we construct sets $A$ with subquadratically many distinct differences such that one of the `worst' $k$-configurations that occurs in our construction --- i.e., one with minimal $d(\cC)$ --- actually has to occur in \emph{every} set $A$ with subquadratically many distinct differences; this is what allows us to determine the quadratic threshold. 
\end{remark}

When $k$ is odd, we will use Proposition \ref{prop:odd-upper-bound} (in place of Theorem \ref{thm:upper-bound}) to obtain the bound of $o(n^2)$. Meanwhile, to obtain the bound of $\Omega(n^2)$, we extend the bound from Theorem \ref{thm:lower-bound} to odd $k$ using the following simple observation. 

\begin{fact}\label{fact:monotonicity-k}
    For all $n \geq k \geq 2$ and $\ell$, we have $g(n, k, \ell) \geq g(n, k - 1, \ell - k + 1)$. 
\end{fact}

\begin{proof}
    If a set $A \subseteq \RR$ satisfies the $(k, \ell)$-local property, then we claim that it must also satisfy the $(k - 1, \ell - k + 1)$-local property. To see this, for any $(k - 1)$-element subset $A' \subseteq A$, choose any $a \in A \setminus A'$. Then $A' \cup \{a\}$ has $k$ elements, so by assumption it contains at least $\ell$ distinct differences. Then removing $a$ removes at most $k - 1$ of these differences, so $A'$ contains at least $\ell - k + 1$ distinct differences, as desired. 

    Then since $g(n, k, \ell)$ is defined as the \emph{minimum} value of $\abs{A - A}$ over all sets $A$ satisfying the $(k, \ell)$-local property, it follows that $g(n, k, \ell) \geq g(n, k - 1, \ell - k + 1)$.
\end{proof}

Now for odd $k$, we can deduce Corollary \ref{cor:odd-qt} as well:
\begin{itemize}
    \item Proposition \ref{prop:odd-upper-bound} immediately gives that \[g\left(n, k, \frac{(k + 1)^2}{4} - 4\right) = o(n^2).\] 
    \item Applying Theorem \ref{thm:lower-bound} with $t = 1$ to $k - 1$ (which is even) and using Fact \ref{fact:monotonicity-k} gives that \[g\left(n, k, \frac{(k + 1)^2}{4}\right) \geq g\left(n, k - 1, \frac{(k + 1)^2}{4} - k + 1\right) = g\left(n, k - 1, \frac{(k - 1)^2}{4} + 1\right) = \Omega(n^2).\] 
\end{itemize}
In particular, these bounds mean that the quadratic threshold for odd $k$ is between $\frac{(k + 1)^2}{4} - 3$ and $\frac{(k + 1)^2}{4}$. 

\begin{remark}
    Similarly to Remark \ref{rmk:why-qt}, we now comment on why the values of $\ell$ in these two bounds differ by $4$ (or equivalently, why our bounds on the quadratic threshold for odd $k$ differ by $3$). Let $k = 2s + 1$. Then the proof of our bound $g(n, k, \ell) = \Omega(n^2)$ can be viewed as showing that if $\abs{A - A}$ is small compared to $n^2$, then the $k$-configuration \[\cC_{\text{sum}} = \{x_1 + x_{s + 1} = x_2 + x_{s + 2} = \cdots = x_s + x_{2s}\}\] (in which there is no equation involving $x_{2s + 1}$) must occur in $A$, and therefore the value of $\ell$ in the bound is $d(\cC_{\text{sum}}) + 1$. (This is not how we phrased the argument, but it is equivalent in some sense, and gives the same value of $\ell$.) 
    
    On the other hand, we again have a bound of $g(n, k, \ell) = o(n^2)$ where $\ell$ is the minimum value of $d(\cC)$ over all $2$-good $k$-configurations $\cC$. As stated in Remark \ref{rmk:odd-bounding-equality}, for $k \geq 7$ this minimum is achieved by \[\cC_{\text{sum}}^+ = \{x_1 + x_{s + 1} = x_2 + x_{s + 2} = \cdots = x_s + x_{2s}, \, x_1 - x_2 - x_3 + x_{2s + 1} = 0\}\] (since $\cC_{\text{sum}}^+$ is an equality case of Lemma \ref{lem:odd-bounding-diffs}, our bound on $d(\cC)$ for $2$-good $\cC$), so the value of $\ell$ we get in this bound is $d(\cC_{\text{sum}}^+)$. The additional equation in $\cC_{\text{sum}}^+$ compared to $\cC_{\text{sum}}$ means we have $d(\cC_{\text{sum}}^+) = d(\cC_{\text{sum}}) - 3$, as we gain the equalities $x_{2s + 1} - x_2 = x_3 - x_1$, $x_{2s + 1} - x_3 = x_2 - x_1$, and $x_{2s + 1} - x_{s + 1} = x_3 - x_{s + 2}$ (or $x_{2s + 1} - x_{s + 1} = x_2 - x_{s + 3}$); this causes the two values of $\ell$ to differ by $4$. 

    In fact, for $k = 3$ and $5$, one can show that the minimum value of $d(\cC)$ over $2$-good $k$-configurations $\cC$ is achieved by $\cC_{\text{sum}}$ (for $k = 3$, a valid and AP-free $3$-configuration cannot contain any difference equalities; for $k = 5$, by Lemma \ref{lem:2-eqns-5-vars} a valid and AP-free $5$-configuration contains at most one difference equality), and therefore the quadratic threshold is exactly $d(\cC_{\text{sum}}) + 1 = \frac{(k + 1)^2}{4}$. However, determining the exact quadratic threshold for odd $k \geq 7$ remains an interesting open problem. 
\end{remark}

\begin{remark}
    It is also natural to ask how $g(n, k, \ell)$ behaves for $\ell$ immediately below the quadratic threshold. For even $k$, Theorem \ref{thm:upper-bound} in fact gives that $g(n, k, \frac{k^2}{4}) = o(n^c)$ for all $c > 2 - \frac{2}{k}$ --- for all such $c$ we have $\frac{(c - 1)(k - 1)}{c} > \frac{k}{2} - 1$, so the value of $\ell$ for which Theorem \ref{thm:upper-bound} gives a bound of $o(n^c)$ is still $\frac{k^2}{4}$. On the other hand, the best known lower bound on $g(n, k, \frac{k^2}{4})$ is the bound \eqref{eqn:fps} due to Fish, Pohoata, and Sheffer \cite{FLS19} with $r = 3$, which gives that $g(n, k, \frac{k^2}{4}) = \Omega(n^{3/2 - 9/k})$ for all $k \geq 18$ divisible by $6$.

    These two bounds have very different behavior --- the exponent $2 - \frac{2}{k}$ in the upper bound grows arbitrarily close to $2$ as $k \to \infty$, while the exponent in the lower bound is always at most $\frac{3}{2}$, and is therefore bounded away from $2$. It would be interesting to understand which behavior is `correct' --- does there exist a constant $\eps > 0$ not depending on $k$ such that for all $k$, for all $\ell$ we either have $g(n, k, \ell) = \Omega(n^2)$ or $g(n, k, \ell) = O(n^{2 - \eps})$? (It would also be interesting to answer this question for $k$ restricted to any infinite set, e.g., for all even $k$.) 
\end{remark}

\subsection{Intermediate bounds}

For other exponents $1 < c < 2$, the values of $\ell$ in Theorems \ref{thm:upper-bound} and \ref{thm:lower-bound} for which we can bound $g(n, k, \ell)$ above and below by $n^c$ are farther apart. However, both values of $\ell$ are quadratic in $k$ (considering $c$ and $t$ as constants not depending on $k$ in the two theorems), with leading coefficient tending to $0$ as $c$ tends to $1$. We will use this fact to deduce Corollaries \ref{cor:nc-threshold} and \ref{cor:poly-bounds}. To do so, we first use Fact \ref{fact:monotonicity-k} to extend Theorem \ref{thm:lower-bound} to \emph{all} $k$, rather than just multiples of $2^t$. 

\begin{corollary}\label{cor:lower-for-nonmultiples}
    For all positive integers $t$ and $k$, we have \[g\left(n, k, \frac{3^{t - 1}k^2}{4^t} + 2^tk + \frac{3^{t - 1} + 1}{2}\right) = \Omega(n^{1 + \frac{1}{2^t - 1}}).\] 
\end{corollary}

\begin{proof}
    Let $r$ be the remainder when $k$ is divided by $2^t$, so that $0 \leq r < 2^t$. Then applying Theorem \ref{thm:lower-bound} to $k - r$ (which is divisible by $2^t$) gives 
    \begin{align}
        g\left(n, k - r, \frac{3^{t - 1}(k - r)^2}{4^t} + \frac{3^{t - 1} + 1}{2}\right) = \Omega(n^{1 + \frac{1}{2^t - 1}}).\label{eqn:k-r-lower}
    \end{align} 
    Meanwhile, by repeatedly applying Fact \ref{fact:monotonicity-k}, we have 
    \begin{align}
        g(n, k, \ell) \geq g(n, k - r, \ell - (k - 1) - (k - 2) - \cdots - (k - r)) \geq g(n, k - r, \ell - kr)\label{eqn:lifting-r}
    \end{align} 
    for all $\ell$. Since $r < 2^t$, we have \[\frac{3^{t - 1}k^2}{4^t} + 2^tk + \frac{3^{t - 1} + 1}{2} - kr \geq \frac{3^{t - 1}k^2}{4^t} + \frac{3^{t - 1} + 1}{2} \geq \frac{3^{t - 1}(k - r)^2}{4^t} + \frac{3^{t - 1} + 1}{2},\] so combining \eqref{eqn:k-r-lower} and \eqref{eqn:lifting-r} gives the desired bound. (It is possible to improve the lower-order terms in our expression for $\ell$ by being slightly more careful with these bounds; but only the leading term matters for the following analysis, so we use the above expression for $\ell$ because it is simpler.)
\end{proof}

Using this, we can immediately deduce Corollary \ref{cor:nc-threshold}, which essentially states that for any given constant $c$, the threshold for $\Omega(n^c)$ is quadratic in $k$ (the precise statement is slightly stronger than this). 

\begin{proof}[Proof of Corollary \ref{cor:nc-threshold}]
    First we find a constant $a_1 > 0$ such that $g(n, k, a_1k^2) = o(n^c)$ for large $k$ (providing a lower bound on the threshold for $\Omega(n^c)$ that is quadratic in $k$).
    
    Choose $a_1 = \frac{1}{2}(\frac{c - 1}{c})^2$ (any choice of $a_1$ with $a_1 < (\frac{c - 1}{c})^2$ would work for the same reason; we choose one for concreteness). Then for all sufficiently large $k$ we have \[a_1k^2 \leq \ceil{\frac{(c - 1)(k - 1)}{c}}^2,\] so by Theorem \ref{thm:upper-bound} we have \[g(n, k, a_1k^2) \leq g\left(n, k, \ceil{\frac{(c - 1)(k - 1)}{c}}^2\right) = o(n^c).\] 

    Now we find a constant $0 < a_2 < \frac{1}{2}$ such that $g(n, k, a_2k^2) = \Omega(n^c)$ for large $k$ (providing an upper bound for the threshold for $\Omega(n^c)$ that is quadratic in $k$); note that this automatically implies $a_2 > a_1$. 
    
    Let $t$ be the largest positive integer such that $c \leq 1 + \frac{1}{2^t - 1}$, and choose $a_2 = \frac{1}{2}(\frac{3}{4})^{t + 1}$ (any $a_2 > \frac{3^{t - 1}}{4^t}$ would work for the same reason; we choose this one for concreteness and because it makes the computations in Remark \ref{rmk:nc-threshold-discrep} slightly nicer). Then for all sufficiently large $k$ we have \[a_2k^2 \geq \frac{3^{t - 1}k^2}{4} + 2^tk + \frac{3^{t - 1} + 1}{2},\] so by Corollary \ref{cor:lower-for-nonmultiples} we have \[g(n, k, a_2k^2) \geq g\left(n, k, \frac{3^{t - 1}k^2}{4} + 2^tk + \frac{3^{t - 1} + 1}{2}\right) = \Omega(n^c).\] 
    
    Finally, note that as $c \to 1$, we have $t \to \infty$ and therefore $a_2 = \frac{1}{2}(\frac{3}{4})^{t + 1} \to 0$; since $a_1 < a_2$, this means we have $a_1 \to 0$ as well. 
\end{proof}

\begin{remark}\label{rmk:nc-threshold-discrep}
    We now analyze the discrepancy between the values of $a_1$ and $a_2$ given in the above proof. To do so, we first obtain bounds on $a_2$ purely in terms of $c$ --- note that $c \leq 1 + \frac{1}{2^s - 1}$ if and only if $2^s \leq \frac{c}{c - 1}$, and therefore the value of $t$ in the above proof must satisfy \[2^t \leq \frac{c}{c - 1} < 2^{t + 1}.\] Since we chose $a_2 = \frac{1}{2}(\frac{3}{4})^{t + 1}$, this means we have \[\frac{3}{4}\left(\frac{c - 1}{c}\right)^{\log_2 \frac{4}{3}} \leq 2a_2 < \left(\frac{c - 1}{c}\right)^{\log_2 \frac{4}{3}}.\] Meanwhile, we set $a_1 = \frac{1}{2}(\frac{c - 1}{c})^2$, so we have \[\frac{3}{4}(2a_1)^{\log_4 \frac{4}{3}} \leq 2a_2 < (2a_1)^{\log_4 \frac{4}{3}}.\] 
\end{remark}

We can also prove Corollary \ref{cor:poly-bounds}, which gives upper and lower bounds of the form $n^c$ when $\ell$ is quadratic in $k$, in the same way. 

\begin{proof}[Proof of Corollary \ref{cor:poly-bounds}]
    To obtain the lower bound of $g(n, k, ak^2) = \Omega(n^{c_1})$, let $t$ be the smallest positive integer such that $a > \frac{3^{t - 1}}{4^t}$, and let $c_1 = 1 + \frac{1}{2^t - 1}$. Then for all sufficiently large $k$ we have \[ak^2 \geq \frac{3^{t - 1}}{4^t}k^2 + 2^tk + \frac{3^{t - 1} + 1}{2},\] so by Corollary \ref{cor:lower-for-nonmultiples} we have \[g(n, k, ak^2) \geq g\left(n, k, \frac{3^{t - 1}}{4^t}k^2 + 2^tk + \frac{3^{t - 1} + 1}{2}\right) = \Omega(n^{c_1}).\] 
    
    We now turn to the upper bound. Given $0 < a < \frac{1}{4}$, choose $1 < c_2 < 2$ such that $a < (\frac{c_2 - 1}{c_2})^2$ (this is possible because $(\frac{c - 1}{c})^2$ ranges over $(0, \frac{1}{4})$ as $c$ ranges over $(1, 2)$); it is possible to choose such a $c_2$ as an explicit function of $a$ such that $c_2 \to 0$ as $a \to 0$. Then for all sufficiently large $k$ we have \[ak^2 \leq \ceil{\frac{(c_2 - 1)(k - 1)}{c_2}}^2,\] so by Theorem \ref{thm:upper-bound}, we have \[g(n, k, ak^2) \leq g\left(n, k, \ceil{\frac{(c_2 - 1)(k - 1)}{c_2}}^2\right) = o(n^{c_2}).\] (Note that these two bounds automatically imply that $c_1 < c_2$, so $c_1 \to 1$ as $a \to 0$ as well.) 
\end{proof}

\subsection{On the number of possible exponents}

As stated in Subsection \ref{subsec:results}, for any fixed $k$, we can attempt to understand which exponents `appear' in our asymptotics for $g(n, k, \ell)$ over all $\ell$; to formalize this, we define the set \[S_k = \left\{\liminf_{n \to \infty} \frac{\log g(n, k, \ell)}{\log n} \bigm\vert 0 \leq \ell \leq \binom{k}{2}\right\}.\] 

Clearly $1 \in S_k$ (as $g(n, k, k - 1) = n - 1$), and $2 \in S_k$ for $k \geq 4$ (as $g(n, k, \binom{k}{2}) = \binom{n}{2}$). The currently known bounds on $g(n, k, \ell)$ are not tight enough to allow us to determine any other values that must be in $S_k$. However, we \emph{can} use Theorems \ref{thm:upper-bound} and \ref{thm:lower-bound} to show that $\abs{S_k}$ tends to infinity as $k$ does with a growth rate of at least $\log\log k$, as stated in Corollary \ref{cor:sk-loglogk}. 

\begin{proof}[Proof of Corollary \ref{cor:sk-loglogk}]
    First, it suffices to prove there exists $a$ such that $\abs{S_k} \geq a\log\log k$ for all sufficiently large $k$; then we can adjust $a$ to deal with the finitely many remaining values of $k$ (as $\abs{S_k} \geq 1$ for all $k$). For notational convenience, let \[e(\ell) = \liminf_{n \to \infty} \frac{\log g(n, k, \ell)}{\log n}.\] Since $g(n, k, \ell)$ is (weakly) increasing in $\ell$, the same is true of $e(\ell)$. We wish to find at least $a\log \log k$ values of $\ell$ with distinct values of $e(\ell)$ (for some absolute constant $a > 0$). The idea of the proof is to choose values $\ell_1 > \ell_2 > \cdots$ sufficiently spaced out so that for each $i$, the lower bound on $e(\ell_i)$ given by Theorem \ref{thm:lower-bound} (or rather, Corollary \ref{cor:lower-for-nonmultiples}) is greater than the upper bound on $e(\ell_{i + 1})$ given by Theorem \ref{thm:upper-bound}, which implies that we must have $e(\ell_i) > e(\ell_{i + 1})$. 

    To do so, for each positive integer $t$, let \[\ell^+(t) = \frac{3^{t - 1}k^2}{4^t} + 2^tk + \frac{3^{t - 1} + 1}{2}\] (this is the value of $\ell$ for which Corollary \ref{cor:lower-for-nonmultiples} gives $g(n, k, \ell) = \Omega(n^{1 + \frac{1}{2^t - 1}})$), and let \[\ell^-(t) = \ceil{\frac{k - 1}{2^t}}^2\] (this is the value of $\ell$ for which Theorem \ref{thm:upper-bound} gives $g(n, k, \ell) = o(n^{1 + \frac{1}{2^t - 1}})$). Then we have \[e(\ell^-(t)) \leq 1 + \frac{1}{2^t - 1} \leq e(\ell^+(t))\] for all positive integers $t$. 
    
    We will choose $\ell_i = \ell^+(t_i)$ for certain values of $t_i$. This will automatically give lower bounds on our values of $e(\ell_i)$; the following claim will allow us to obtain comparable upper bounds. 

    \begin{claim}\label{claim:10t-for-exps}
        For any integer $t \geq 2$ such that $k \geq 3 \cdot 4^{10t}$, we have $\ell^+(10t) \leq \ell^-(t + 1)$. 
    \end{claim}

    \begin{proof}
        First, since $k - 1 \geq \frac{k}{2}$, we have \[\ell^-(t + 1) = \ceil{\frac{k - 1}{2^{t + 1}}}^2 \geq \left(\frac{k}{2^{t + 2}}\right)^2 = \frac{1}{4^{t + 2}}k^2.\] Meanwhile, let $s = 10t$, so that \[\ell^+(s) = \frac{3^{s - 1}}{4^s}k^2 + 2^sk + \frac{3^{s - 1} + 1}{2} \leq \frac{3^{s - 1}}{4^s}k^2 + 2^sk + 3^{s - 1}.\] The condition $k \geq 3 \cdot 4^s$ guarantees that the first term is the largest of the three terms (the ratio of the first and second terms is $\frac{3^{s - 1}}{8^s} \cdot k \geq \frac{1}{3}\cdot \left(\frac{3}{8}\right)^s \cdot 3 \cdot 4^s \geq 1$, and the ratio of the first and third is $\frac{k^2}{4^s} \geq 1$), so \[\ell^+(s) \leq 3\cdot \frac{3^{s - 1}}{4^s}k^2 = \left(\frac{3}{4}\right)^sk^2.\] Finally, for all $t \geq 2$, we have \[\frac{1}{4^{t + 2}} \geq \frac{1}{4^{2t}} \geq \left(\frac{3}{4}\right)^{10t}\] (using the fact that $\frac{1}{4} \geq \left(\frac{3}{4}\right)^5$). Combining these bounds gives $\ell^+(10t) \leq \ell^-(t + 1)$, as desired. 
    \end{proof}

    Now let $r = \floor{\log_{10} \log_4 \frac{k}{3}}$. For each integer $1 \leq i \leq r$, define $t_i = 10^i$, and let $\ell_i = \ell^+(t_i)$. Then for all $1 \leq i \leq r$ we have $k \geq 3 \cdot 4^{t_i}$, so Claim \ref{claim:10t-for-exps} implies that for all $1 \leq i \leq r - 1$ we have \[\ell^+(t_{i + 1}) \leq \ell^-(t_i + 1),\] and therefore \[e(\ell^+(t_{i + 1})) \leq e(\ell^-(t_i + 1)) \leq 1 + \frac{1}{2^{t_i + 1} - 1} < 1 + \frac{1}{2^{t_i} - 1} \leq e(\ell^+(t_i)).\] So $e(\ell_r) < e(\ell_{r - 1}) < \cdots < e(\ell_1)$, which means the values of $e(\ell_i)$ for $1 \leq i \leq r$ are all distinct. (Note that for sufficiently large $k$, all our values of $\ell_i$ are valid, i.e., $0 \leq \ell^+(t_i) \leq \binom{k}{2}$ for all our values of $t_i$. To see this, clearly $\ell^+(t) \geq 0$ for all $t$. Meanwhile for sufficiently large $k$ we have $\ell_1 = \ell^+(10) \leq \binom{k}{2}$ (as the leading coefficient in our expression for $\ell^+(10)$ is $\frac{3^9}{4^{10}} < \frac{1}{2}$), and the above argument in particular implies that $\ell_{i + 1} < \ell_i$ for all $1 \leq i \leq r - 1$, so $\ell_i \leq \binom{k}{2}$ for all $1 \leq i \leq r$.)

    This gives us $r$ values $0 \leq \ell \leq \binom{k}{2}$ with distinct values of $e(\ell)$, so \[\abs{S_k} \geq r \geq \log_{10}\log_4 \frac{k}{3} - 1.\] Letting $a = \frac{1}{2\log 10}$, we then have $\abs{S_k} \geq a\log\log k$ for all sufficiently large $k$, as desired. 
\end{proof}


\section*{Acknowledgments} 

This work was done at the University of Minnesota Duluth REU with support from Jane Street Capital, the National Security Agency, and the CYAN Undergraduate Mathematics Fund at MIT. The author thanks Joe Gallian and Colin Defant for organizing the Duluth REU and providing this great research opportunity, Noah Kravitz and Anqi Li for helpful discussions, and Evan Chen, Noah Kravitz, and Mihir Singhal for helpful feedback that improved this paper. The author also thanks the anonymous referee for closely reading the paper and providing helpful comments. 

\bibliographystyle{amsplain}


\begin{dajauthors}
\begin{authorinfo}[sdas]
  Sanjana Das \\
  Massachusetts Institute of Technology \\
  Cambridge, MA \\
  sanjanad\imageat{}mit\imagedot{}edu
\end{authorinfo}
\end{dajauthors}

\end{document}